\documentclass[11pt, reqno]{amsart}

\usepackage[top=3.75cm, bottom=3cm, left=3cm, right=3cm]{geometry}
\usepackage{amsmath}
\usepackage{amsthm}
\usepackage{amsfonts}
\usepackage{amssymb}
\usepackage[all,cmtip]{xy}
\usepackage{hyperref}
\usepackage{color}
\usepackage{eucal}
\usepackage{bm}
\usepackage{enumitem}

\numberwithin{equation}{section}

\theoremstyle{plain}
\newtheorem{theorem}{Theorem}[section]
\newtheorem{lemma}[theorem]{Lemma}
\newtheorem{proposition}[theorem]{Proposition}
\newtheorem{corollary}[theorem]{Corollary}

\theoremstyle{definition}
\newtheorem{definition}[theorem]{Definition}
\newtheorem{remark}[theorem]{Remark}
\newtheorem{question}[theorem]{Question}

\newcommand{\Gr}{\mathrm{Gr}}
\newcommand{\Hilb}{\mathrm{Hilb}}

\newcommand{\Spec}{\operatorname{Spec}}
\newcommand{\GPK}{GPK\textsuperscript{3}}

\newcommand{\Db}{\mathrm{D^b}}

\newcommand{\id}{\mathrm{id}}
\newcommand{\pr}{\mathrm{pr}}
\newcommand{\ttau}{\tilde{\tau}}
\newcommand{\tf}{\tilde{f}}
\newcommand{\ta}{\tilde{a}}

\newcommand{\gl}{\mathfrak{gl}}
\newcommand{\pgl}{\mathfrak{pgl}}
\newcommand{\GL}{\mathrm{GL}}
\newcommand{\PGL}{\mathrm{PGL}}
\newcommand{\tH}{\tilde{H}}
\newcommand{\tG}{\tilde{G}}
\newcommand{\tfh}{\tilde{\fh}}
\newcommand{\tfg}{\tilde{\fg}}

\newcommand{\Aut}{\mathrm{Aut}}
\newcommand{\torfree}{\mathrm{tf}}
\newcommand{\Pic}{\mathrm{Pic}}

\newcommand{\gitq}{/\!\!/}
\newcommand{\act}{\mathrm{act}}
\newcommand{\Ad}{\mathrm{Ad}}

\newcommand{\image}{\mathrm{im}}
\newcommand{\rk}{\operatorname{rk}}
\newcommand{\mult}{\mathrm{mult}}
\newcommand{\tr}{\mathrm{tr}}
\newcommand{\ch}{\mathrm{ch}}

\newcommand{\RGamma}{\mathrm{R}\Gamma}

\newcommand{\Var}{\mathrm{Var}}
\newcommand{\vepsilon}{\varepsilon}

\newcommand{\st}{\;\vline\;} % 'such that' set notation bar
\newcommand{\set}[1]{\left\{#1\right\}}

\newcommand{\bA}{\mathbf{A}}
\newcommand{\bF}{\mathbf{F}}
\newcommand{\bZ}{\mathbf{Z}}
\newcommand{\bP}{\mathbf{P}}
\newcommand{\bPv}{\mathbf{P}^{\vee}}
\newcommand{\bQ}{\mathbf{Q}}

\newcommand{\bL}{\mathbf{L}}

\newcommand{\cO}{\mathcal{O}}

\newcommand{\cQ}{\mathcal{Q}}
\newcommand{\cI}{\mathcal{I}}
\newcommand{\cE}{\mathcal{E}}
\newcommand{\cF}{\mathcal{F}}
\newcommand{\cM}{\mathcal{M}}
\newcommand{\cN}{\mathcal{N}}
\newcommand{\cR}{\mathcal{R}}
\newcommand{\cU}{\mathcal{U}}
\newcommand{\cK}{\mathcal{K}}
\newcommand{\cX}{\mathcal{X}}
\newcommand{\cY}{\mathcal{Y}}
\newcommand{\cZ}{\mathcal{Z}}

\newcommand{\fg}{\mathfrak{g}}
\newcommand{\fh}{\mathfrak{h}}

\newcommand{\rT}{\mathrm{T}}
\newcommand{\rK}{\mathrm{K}}
\newcommand{\rH}{\mathrm{H}}

\newcommand{\rN}{\mathrm{N}}
\newcommand{\rO}{\mathrm{O}}
\newcommand{\rQ}{\mathrm{Q}}
\newcommand{\rS}{\mathrm{S}}
\newcommand{\rd}{\mathrm{d}}

\begin{document}

\title{Intersections of two Grassmannians in $\bP^9$}

\author{Lev A. Borisov}
\address{Department of Mathematics\\
Rutgers University\\
Piscataway, NJ 08854} \email{borisov@math.rutgers.edu}

\author{Andrei C\u ald\u araru}
\address{Department of Mathematics \\ 
University of Wisconsin--Madison \\ 
Madison, WI 53706}
\email{andreic@math.wisc.edu}

\author{Alexander Perry}
\address{Department of Mathematics \\ 
Columbia University \\ 
New York, NY 10027}
\email{aperry@math.columbia.edu}

\thanks{
L.B. was partially supported by NSF Grants DMS-1201466 and DMS-1601907. 
A.C. was partially supported by NSF grant DMS-1200721. 
A.P. was partially supported by an NSF postdoctoral fellowship, DMS-1606460.}

\begin{abstract}
We study the intersection of two copies of $\Gr(2,5)$ embedded 
in $\bP^9$, and the intersection of the two projectively dual Grassmannians 
in the dual projective space.  
These intersections are deformation equivalent,  
derived equivalent Calabi--Yau threefolds. 
We prove that generically they are not birational. 
As a consequence, we obtain a counterexample to the birational Torelli problem for 
Calabi--Yau threefolds. 
We also show that these threefolds give a new pair of varieties whose 
classes in the Grothendieck ring of varieties are not equal, but whose difference 
is annihilated by a power of the class of the affine line. 
Our proof of non-birationality 
involves a detailed study of the moduli stack of 
Calabi--Yau threefolds of the above type, which may be of independent interest. 
\end{abstract}

\maketitle

%%%%%%%%%%%%%%%%%%%%%%%%%%%%%%%%%%%%%%%

\section{Introduction}
\label{section-intro}
Let $k$ be an algebraically closed field of characteristic $0$. 
Let $W$ be a $10$-dimensional vector space over $k$, 
whose projectivization we denote by $\bP = \bP(W)$.   
Let $V$ be a $5$-dimensional vector space over $k$, 
together with isomorphisms 
\begin{equation*}
\label{gr-data}
\phi_i \colon \wedge^2V \xrightarrow{\sim} W, ~ i=1,2.
\end{equation*} 
By composing the Pl\"{u}cker embedding with the resulting 
isomorphisms $\bP(\wedge^2V) \cong \bP$, we obtain two embeddings 
$\Gr(2,V) \hookrightarrow \bP$, whose images we denote by $\Gr_i \subset \bP$.    
For generic $\phi_i$, the intersection 
\begin{equation}
\label{X}
X = \Gr_1 \cap \Gr_2 \subset \bP
\end{equation}
is a smooth Calabi--Yau threefold (i.e. $\omega_X \cong \cO_X$ and $\rH^j(X, \cO_X) = 0$ for $j = 1,2$) 
with Hodge numbers 
\begin{equation*}
\label{hodge-numbers}
h^{1,1}(X) = 1, \quad h^{1,2}(X) = 51. 
\end{equation*} 
These varieties first appeared in work of Gross and Popescu~\cite{gross-popescu}. 
Later G.~Kapustka~\cite{G-Kapustka} used geometric transitions to construct 
Calabi--Yau threefolds with the above Hodge numbers, which were shown 
by M.~Kapustka~\cite{M-Kapustka} to be isomorphic to Grassmannian intersections of 
the above form. 
Independently, Kanazawa~\cite{kanazawa} gave a direct computation of the 
Hodge numbers of such Grassmannian intersections. 
After these authors, we call $X$ as above a \emph{{\GPK} threefold}. 

The isomorphisms $\phi_i$ naturally determine another {\GPK} threefold, as follows. 
We write \mbox{$\bP^\vee = \bP(W^\vee)$} for the dual projective space. 
Then the induced isomorphisms 
\begin{equation*}
(\phi_i^{-1})^* \colon \wedge^2 V^\vee \xrightarrow{\sim} W^\vee, ~ i =1,2, 
\end{equation*}
correspond to two embeddings $\Gr(2,V^\vee) \hookrightarrow \bP^\vee$, 
whose images we denote by $\Gr_i^\vee \subset \bP^\vee$. 
As the notation suggests, $\Gr_i^{\vee}$ is the projective dual of $\Gr_i$ 
(see Remark~\ref{remark-projective-dual}). 
We consider the intersection 
\begin{equation}
\label{Y}
Y = \Gr_1^\vee \cap \Gr_2^\vee \subset \bP^\vee. 
\end{equation}
If $X$ is a smooth threefold, then $Y$ is too (Lemma~\ref{X-Y-smooth}).  
In this case, $X$ and $Y$ are thus smooth deformation equivalent Calabi--Yau threefolds, 
which we call \emph{{\GPK} double mirrors}. 
This terminology is justified by the following result, 
which should be thought of as saying $X$ and $Y$ ``have the same mirror''. 

\begin{theorem}[{\cite[Theorem 6.3]{categorical-joins}}]
\label{theorem-equivalence}
If $X$ and $Y$ are of the expected dimension $3$ (but possibly singular), 
then there is an equivalence 
\begin{equation*}
\Db(X) \simeq \Db(Y)
\end{equation*}
of bounded derived categories of coherent sheaves. 
\end{theorem}

Our main result says that, nonetheless, $X$ and $Y$ are typically not birational. 

\begin{theorem}
\label{theorem-not-birational}
For generic isomorphisms $\phi_i$, 
the varieties $X$ and $Y$ are not birational. 
\end{theorem}

\begin{remark}
John Ottem and J{\o}rgen Rennemo \cite{jj-torelli} have also 
independently proved Theorem~\ref{theorem-not-birational}. 
\end{remark}

Since $X$ and $Y$ have Picard number $1$, the conclusion of 
Theorem~\ref{theorem-not-birational} is equivalent to $X$ and $Y$ being non-isomorphic. 
We prove this by an infinitesimal argument, summarized at the 
end of \S\ref{subsection-geometry-KCY} below. 

Generic {\GPK} double mirrors appear to give the first example of 
deformation equivalent, derived equivalent, but non-birational Calabi--Yau threefolds. 
We note that there are several previously known examples of 
derived equivalent but non-birational Calabi--Yau threefolds: 
the Pfaffian--Grassmannian pair \cite{pfaffian-grassmannian, kuznetsov-HPD-lines}, 
the Gross--Popescu pair \cite{bak, schnell}, 
the Reye congruence and double quintic symmetroid pair \cite{reye-pair}, 
and the $G_2$-Grassmannian pair \cite{kuznetsov-G2-grassmannian}. 
In these examples, the varieties in question are 
not deformation equivalent and are easily seen to be non-birational.  

\subsection{The birational Torelli problem}
One of our motivations for proving Theorem~\ref{theorem-not-birational} was 
the birational Torelli problem for Calabi--Yau threefolds, which asks the following. 

\begin{question}
\label{question-birational-torelli}
If $M_1$ and $M_2$ are smooth deformation equivalent 
complex Calabi--Yau threefolds such that there is an isomorphism 
$\rH^3(M_1,\bZ)_{\mathrm{\torfree}} \cong \rH^3(M_2,\bZ)_{\torfree}$
of polarized Hodge structures, then are $M_1$ and $M_2$ birational? 
\end{question}

Here, for an abelian group $A$, we write $A_{\torfree}$ for the quotient 
by its torsion subgroup. 
As observed in~\cite[Page 857, footnote]{addington}, 
if $M_1$ and $M_2$ are derived equivalent Calabi--Yau threefolds,  
then there is an isomorphism $\rH^3(M_1,\bZ)_{\torfree} \cong \rH^3(M_2,\bZ)_{\torfree}$ 
 of polarized Hodge structures.  
(See also \cite[Proposition 3.1]{caldararu-hodge}  
where up to inverting $2$ such an isomorphism is shown.) 
In particular, any pair of deformation equivalent, derived equivalent, but non-birational 
complex Calabi--Yau threefolds 
gives a negative answer to Question~\ref{question-birational-torelli}. 
Hence together Theorems~\ref{theorem-equivalence} and~\ref{theorem-not-birational} give 
the following. 

\begin{corollary}
\label{corollary-birational-torelli}
Generic complex {\GPK} double mirrors give a counterexample to the 
birational Torelli problem for Calabi--Yau threefolds. 
\end{corollary}

Previously, Szendr{\H o}i~\cite{szendroi-torelli} showed the usual Torelli problem fails for 
Calabi--Yau threefolds, i.e. the answer to Question~\ref{question-birational-torelli} is negative if 
``birational'' is replaced with ``isomorphic''. 
As far as we know, the birational version was open until now. 
For earlier work on this problem, see~\cite{szendroi-birational-torelli, caldararu-hodge}. 

\subsection{The Grothendieck ring of varieties}
A second motivation for this work was the problem of producing nonzero classes in the Grothendieck 
ring $\rK_0(\Var/k)$ of $k$-varieties which are annihilated by a power of the class 
$\bL = [\bA^1]$ of the affine line. 
Recall that $\rK_0(\Var/k)$ is defined as the free abelian group on 
isomorphism classes $[Z]$ of algebraic varieties $Z$ over $k$ modulo the relations 
\begin{equation*}
[Z] = [U] + [Z \setminus U] \quad \text{for all open subvarieties $U \subset Z$}, 
\end{equation*} 
with product induced by products of varieties. 
In~\cite{borisov-zero-divisor} the first author used the Pfaffian--Grassmannian pair   
of Calabi--Yau $3$-folds to show that $\bL$ is a zero divisor in $\rK_0(\Var/k)$. 
This sparked a flurry of results, which 
show that for a number of pairs of derived equivalent Calabi--Yau varieties  
$(M_1, M_2)$, we have $[M_1] \neq [M_2]$ but $([M_1] - [M_2])\bL^r = 0$ for 
some positive integer $r$. 
Namely, this holds for the Pfaffian--Grassmannian pair~\cite{martin} 
(refining~\cite{borisov-zero-divisor}), 
the $G_2$-Grassmannian pair~\cite{G2-grassmannian-zero-divisor}, 
certain pairs of degree $12$ K3 surfaces~\cite{hassett-degree-12, ueda-degree-12}, 
and certain pairs $(M_1,M_2)$ where $M_1$ is a degree $8$ K3 surface and 
$M_2$ a degree $2$ K3 surface~\cite{kuznetsov-L-equivalence}. 
We prove that {\GPK} double mirrors give another such example. 

\begin{theorem}
\label{theorem-L-equivalence}
If $X$ and $Y$ are {\GPK} double mirrors, then 
\begin{equation*}
([X]-[Y])\bL^4 = 0. 
\end{equation*}
Moreover, if the isomorphisms $\phi_i$ defining $X$ and $Y$ are generic, then $[X] \neq [Y]$. 
\end{theorem}

The first statement is proved by studying a certain 
incidence correspondence, and 
the second statement follows from Theorem~\ref{theorem-not-birational} 
by an argument from~\cite{borisov-zero-divisor}. 
We note that Theorem~\ref{theorem-L-equivalence} verifies a case of the ``D-equivalence 
implies L-equivalence'' conjecture of~\cite{kuznetsov-L-equivalence} 
(see also \cite{ueda-degree-12}).

\subsection{Geometry of {\GPK} threefolds and their moduli} 
\label{subsection-geometry-KCY}
Along the way to Theorem~\ref{theorem-not-birational}, we 
prove a number of independently interesting results on the 
geometry of {\GPK} threefolds and their moduli. 

For $X$ a fixed {\GPK} threefold as in~\eqref{X}, 
we prove the two Grassmannians $\Gr_1$ and $\Gr_2$ 
containing $X$ are unique (Proposition~\ref{proposition-unique-Gr}), 
and use this to explicitly describe the automorphism group of $X$ 
(Lemma~\ref{lemma-aut-group}). 

In terms of moduli, we consider the open subscheme $U$ of the moduli 
space of pairs of embedded Grassmannians $\Gr_1 , \Gr_2 \subset \bP$ 
such that $X = \Gr_1 \cap \Gr_2$ is a smooth threefold. 
The group $\bZ/2 \times \PGL(W)$ acts on $U$ (with $\bZ/2$ swapping the 
Grassmannians), 
and we define the \emph{moduli stack of {\GPK} data} as the 
quotient $\cN = [(\bZ/2 \times \PGL(W)) \backslash U]$. 
Let $\cM$ be the \emph{moduli stack of {\GPK} threefolds}, defined as 
a $\PGL(W)$-quotient of an open subscheme of the appropriate Hilbert 
scheme. 
The morphism $U \to \cM$ given pointwise 
by $(\Gr_1, \Gr_2) \mapsto \Gr_1 \cap \Gr_2$ 
descends to a morphism 
$f \colon \cN \to \cM$, 
which we call the \emph{$\PGL$-parameterization of $\cM$}. 
Our main moduli-theoretic results are the following. 

\begin{theorem}
\label{theorem-moduli} 
The $\PGL$-parameterization $f \colon \cN \to \cM$ is an open immersion of 
smooth separated Deligne--Mumford stacks of finite type over $k$. 
\end{theorem} 

\begin{theorem}
\label{theorem-faithful-action}
Let $s \in \cN$ be a geometric point. 
Then the automorphism group of $s$ acts faithfully on the tangent 
space $\rT_{s} \cN$, i.e. the homomorphism 
$\Aut_{\cN}(s) \to \GL(\rT_{s}\cN)$ is injective. 
\end{theorem}

\begin{corollary}
A generic {\GPK} threefold has trivial automorphism group. 
\end{corollary}

\begin{proof}
The stack $\cN$ is irreducible by construction, and smooth and Deligne--Mumford by Theorem~\ref{theorem-moduli}. 
It is well-known that in this situation, the generic point of $\cN$ has trivial automorphism 
group if and only if the automorphism groups of geometric points act faithfully on tangent 
spaces.  
\end{proof} 

The operation $(\Gr_1, \Gr_2) \mapsto (\Gr_1^{\vee}, \Gr_2^{\vee})$ 
descends to an involution $\tau \colon \cN \to \cN$, which we call the 
\emph{double mirror involution}. 
In the above terms, our proof of Theorem~\ref{theorem-not-birational} 
boils down to the following statement: 
there exists a fixed point $s \in \cN$ of $\tau$ such that the derivative 
$\rd_{s} \tau \in \GL(\rT_{s}\cN)$ is not contained 
in the image of the homomorphism 
$\Aut_{\cN}(s) \to \GL(\rT_{s}\cN)$. 
For this, we use our description of the automorphism groups of 
{\GPK} threefolds to show the traces of involutions in the image 
of $\Aut_{\cN}(s) \to \GL(\rT_{s}\cN)$ are contained in an explicit finite list 
(Proposition~\ref{proposition-tr-involution}), and then we exhibit a fixed 
point $s \in \cN$ of $\tau$ such that $\tr(\rd_{s}\tau)$ does not occur 
in this list (Lemma~\ref{lemma-trace-tau}). 
 
\subsection{Organization of the paper}
In \S\ref{section-geometry-KCYs}, we prove the results on the geometry of a 
fixed {\GPK} threefold described above. 
In \S\ref{section-moduli}, we construct the moduli stacks $\cM$ and $\cN$ 
of {\GPK} threefolds and {\GPK} data, and prove Theorem~\ref{theorem-moduli}. 
In \S\ref{section-infinitesimal-computations}, we prove our results 
on the action of automorphism groups on tangent spaces (Theorem~\ref{theorem-faithful-action} 
and Proposition~\ref{proposition-tr-involution}). 
In \S\ref{section-tau}, we show that the operation of passing to the double 
mirror preserves smoothness of {\GPK} threefolds, use this to define 
the double mirror involution $\tau$ of $\cN$, and compute the derivative of $\tau$. 
In \S\ref{section-theorem-not-birational} we prove Theorem~\ref{theorem-not-birational}. 
In \S\ref{section-L-equivalence} we prove Theorem~\ref{theorem-L-equivalence}. 
Finally, in Appendix~\ref{appendix} we gather some Borel--Weil--Bott computations 
which are used in the main body of the paper. 

\subsection{Notation} 
We work over an algebraically closed ground field $k$ of characteristic $0$. 
As above, $V$ and $W$ denote fixed $k$-vector spaces of dimensions $5$ and $10$, 
$\bP = \bP(W)$, and $\bPv = \bP(W^{\vee})$. 
We fix an isomorphism $\phi \colon \wedge^2 V \xrightarrow{\sim} W$, 
and let $\Gr \subset \bP$ denote the corresponding embedded $\Gr(2,V)$. 
Further, we set $G = \PGL(W)$ and $H = \PGL(V)$, and denote by 
$\fg$ and $\fh$ their Lie algebras; 
there are embeddings $H \to G$ and $\fh \to \fg$ 
by virtue of the isomorphism~$\phi$. 
Given a variety $Z$ with a morphism to $\bP$, we write $\cO_Z(1)$ 
for the pullback of $\cO_{\bP}(1)$. 

\subsection{Acknowledgements} 
We are grateful to Johan de Jong for very useful conversations about this work. 
We also benefited from discussions with Ron Donagi, Sasha Kuznetsov, and Daniel Litt. 
We thank Micha{\l} Kapustka for interesting comments and for informing us about 
the history of {\GPK} threefolds. 
We thank John Ottem and J{\o}rgen Rennemo for coordinating the release 
of their paper with ours.

%%%%%%%%%%%%%%%%%%%%%%%%%%%%%%%%%%%%%%%%%%

\section{Geometry of {\GPK} threefolds} 
\label{section-geometry-KCYs}

In this section, we show that a {\GPK} threefold is contained in a unique pair of Grassmannians 
in $\bP$ (Proposition~\ref{proposition-unique-Gr}). 
The key ingredient for this is the stability of the restrictions of the normal bundles 
of these Grassmannians (Proposition~\ref{proposition-Ni-stable}). 
As a consequence, we obtain an explicit description of the automorphism groups 
of {\GPK} threefolds (Lemma~\ref{lemma-aut-group}). 

\subsection{The Grassmannians containing a {\GPK} threefold} 
Recall that if $X$ is a smooth $n$-dimensional projective variety with an ample 
divisor $H$, then the slope of a torsion free sheaf $\cE$ on $X$ is defined by 
\begin{equation*}
\mu(\cE) = \frac{c_1(\cE) \cdot H^{n-1}}{\rk(\cE)}. 
\end{equation*}
Note that $c_1(\cE)$ can be computed as the first Chern class of the line bundle 
$\det(\cE) = ((\wedge^{r} \cE)^{\vee})^{\vee}$, where $r = \rk(\cE)$.  
The sheaf $\cE$ is called \emph{slope stable} 
if for every subsheaf $\cF \subset \cE$ such that $0 < \rk(\cF) < \rk(\cE)$, we have 
\begin{equation*}
\mu(\cF) < \mu(\cE). 
\end{equation*}

If $X \subset \bP$ is a {\GPK} threefold, we set $H = c_1(\cO_X(1))$.  

\begin{proposition}
\label{proposition-Ni-stable}
Let $X = \Gr_1 \cap \Gr_2 \subset \bP$ be a {\GPK} threefold, 
and let $\rN_i = \rN_{\Gr_i/\bP}$ be the normal bundle of $\Gr_i \subset \bP$. 
Then $\rN_i \vert_X$ is slope stable. 
\end{proposition}

\begin{proof}
By Lemma~\ref{lemma-normal-bundle} there is an isomorphism 
$\rN_i \cong \cQ_i^\vee(2)$, where $\cQ_i$ is the tautological rank~$3$ quotient bundle. 
Hence it suffices to prove that the bundle $\cQ_i \vert_X$ is stable.
Let $\cF \subset \cQ_i \vert_X$ be a subsheaf of rank $r=1$ or $2$. 
Since $H$ generates $\Pic(X)$ (see~\cite{G-Kapustka}), we can write 
$c_1(\cF) = tH$ for some $t \in \bZ$. 
Then taking the $r$-th exterior power of the inclusion $\cF \subset \cQ_i \vert_X$ and passing to 
double duals, we get a nonzero section of $\wedge^r(\cQ_i \vert_X)(-tH)$. 
Hence $t \leq 0$ by Lemma~\ref{lemma-Ni-vanishing}. We conclude  
\begin{equation*}
\mu(\cF) = tH^3/r < H^3/3 = \mu(\cQ_i|_X). \qedhere
\end{equation*} 
\end{proof}
 
The following result shows the representation of a {\GPK} threefold as an 
intersection of two Grassmannians is unique. 

\begin{lemma}
\label{lemma-unique-Gr-intersection}
Let $X \subset \bP$ be a {\GPK} threefold.  
Assume $\phi_i \colon \wedge^2 V \xrightarrow{\sim} W$, $1 \leq i \leq 4$, 
are isomorphisms whose corresponding Grassmannian embeddings $\Gr_i \subset \bP$ satisfy 
\begin{equation*} 
X = \Gr_1 \cap \Gr_2 = \Gr_3 \cap \Gr_4. 
\end{equation*}
Then either $\Gr_1 = \Gr_3$ and $\Gr_2 = \Gr_4$, or $\Gr_1 = \Gr_4$ and $\Gr_2 = \Gr_3$.  
\end{lemma}

\begin{proof}
As above, let $\rN_i = \rN_{\Gr_i/\bP}$. 
The restrictions $\rN_i\vert_X$ all have the same slope, and are stable by Proposition~\ref{proposition-Ni-stable}. 
Hence any morphism $\rN_i\vert_X \to \rN_j\vert_X$ is either zero or an isomorphism. 
Considering the inclusion $\rN_i \vert_X \subset \rN_{X/\bP}$, $i=1,2$, followed  
by projection onto the summands of the decomposition $\rN_{X/\bP} \cong \rN_3\vert_X \oplus \rN_4\vert_X$, 
we conclude that either $\rN_1\vert_X \cong \rN_3 \vert_X$ and $\rN_2 \vert_X \cong \rN_4 \vert_X$, 
or $\rN_1 \vert_X \cong \rN_4 \vert_X$ and $\rN_2 \vert_X \cong \rN_3 \vert_X$. 
Hence to finish, it suffices to show that the isomorphism class of $\rN_i|_X$ determines 
$\Gr_i \subset \bP$. 

By Lemma~\ref{lemma-normal-bundle}, $\rN_i\vert_X$ determines the restriction $\cQ_i\vert_X$ 
of the tautological rank $3$ quotient bundle via $(\rN_i\vert_X)^\vee(2) \cong \cQ_i\vert_X$. 
The inclusion $\Gr_i \subset \bP$ is determined by $\cQ_i$ as follows: $V \cong \rH^0(\Gr_i, \cQ_i)$,  
taking the third exterior power induces an isomorphism $\wedge^3 V \cong W^\vee$ 
(note that $\wedge^3 \cQ_i = \cO(1)$), and the dual isomorphism $\wedge^3V^\vee \cong W$ is identified 
with $\phi_i$ under the isomorphism $\wedge^2 V \cong \wedge^3 V^\vee$. 
But the restriction maps $V \cong \rH^0(\Gr_i, \cQ_i) \to \rH^0(X, \cQ_i\vert_X)$ and 
$W^\vee \cong \rH^0(\bP, \cO_\bP(1)) \to \rH^0(X, \cO_X(1))$ are isomorphisms 
by Lemma~\ref{lemma-restriction-Qi}. 
The isomorphism class of $\rN_i\vert_X$ thus determines 
$\phi_i \colon \wedge^2 V \xrightarrow{\sim} W$ 
up to an isomorphism of $V$, and hence determines the Grassmannian $\Gr_i$. 
\end{proof}

As a slight strengthening of Lemma~\ref{lemma-unique-Gr-intersection}, we prove 
that a {\GPK} threefold is contained in a unique pair of Grassmannians. 
We note, however, that Lemma~\ref{lemma-unique-Gr-intersection}  
already suffices for our purposes in this paper. 

\begin{proposition}
\label{proposition-unique-Gr}
Let $X = \Gr_1 \cap \Gr_2 \subset \bP$ be a {\GPK} threefold.  
Let $\Gr_3 \subset \bP$ be the image of an embedding $\Gr(2,V) \hookrightarrow \bP$ 
given by an isomorphism $\phi_3 \colon \wedge^2 V \xrightarrow{\sim} W$. 
If $X \subset \Gr_3$, then either $\Gr_3 = \Gr_1$ or $\Gr_3 = \Gr_2$. 
\end{proposition}

\begin{proof}
Let $\rN_i = \rN_{\Gr_i/\bP}$ for $i=1,2,3$. 
We have an injective morphism 
\begin{equation*}
(\rN_3 \vert_X)^{\vee} \to (\rN_1 \vert_X)^{\vee} \oplus (\rN_2 \vert_X)^{\vee}. 
\end{equation*}
We claim that one of the morphisms $\alpha_1 \colon (\rN_3 \vert_X)^{\vee} \to (\rN_1 \vert_X)^{\vee}$ 
or $\alpha_2 \colon (\rN_3 \vert_X)^{\vee} \to (\rN_2 \vert_X)^{\vee}$ obtained by composition with 
the projections is an isomorphism. 
Since the $(\rN_i \vert_X)^{\vee}$ all have the same rank and determinant, 
$\alpha_k$ is an isomorphism if it is injective. 
Hence it suffices to show that either  
$\cK = \ker(\alpha_2) \subset (\rN_1 \vert_X)^{\vee}$ or  
$\cI = \image(\alpha_2) \subset (\rN_2 \vert_X)^{\vee}$ vanishes. 
If $\mu$ is the common slope of the $(\rN_i \vert_X)^{\vee}$, then 
either $\mu(\cK) \geq \mu$ or $\mu(\cI) \geq \mu$. 
Since $(\rN_1 \vert_X)^{\vee}$ and $(\rN_2 \vert_X)^{\vee}$ are slope stable 
by Proposition~\ref{proposition-Ni-stable}, it follows that either $\cK$ or $\cI$ 
vanishes. 

So we may assume $\alpha_1 \colon (\rN_3 \vert_X)^{\vee} \to (\rN_1 \vert_X)^{\vee}$ is 
an isomorphism. 
Note that this means $\Gr_2$ and $\Gr_3$ intersect transversely along $X$. 
We will show that $X = \Gr_2 \cap \Gr_3$, which by Lemma~\ref{lemma-unique-Gr-intersection}  
proves the proposition. 
The intersection $\Gr_2 \cap \Gr_3$ consists of components of dimension at least $3$. 
If there are no components of dimension at least $4$, then $X$ and $\Gr_2 \cap \Gr_3$ have 
the same degree in $\bP$, forcing the inclusion $X \subset \Gr_2 \cap \Gr_3$ to be an equality. 
Hence it suffices to show there are no components of dimension at least~$4$. 
By the transversality of $\Gr_2$ and $\Gr_3$ along $X$, the components of the 
intersection $\Gr_2 \cap \Gr_3$ which are not equal to $X$ must be disjoint from $X$. 
But by Lemma~\ref{lemma-class-X-Gr} the class of $X$ in the Chow ring of 
$\Gr_2$ is $5H^3$, which implies $X$ intersects nontrivially any effective 
cycle in $\Gr_3$ of dimension at least $4$. 
\end{proof}

\subsection{Automorphism groups} 
\label{subsection-aut-group}
{\GPK} threefolds can alternatively be described as intersections of 
translates of a fixed Grassmannian. 
Namely, fix an isomorphism $\phi \colon \wedge^2V \xrightarrow{\sim} W$. 
Let $\Gr \subset \bP$ denote the corresponding embedded Grassmannian $\Gr(2,V)$.  
Let $G = \PGL(W)$. 
Then for any $(g_1, g_2) \in G \times G$ we set 
\begin{equation}
\label{Xg1g2}
X_{g_1,g_2} = g_1 \Gr \cap g_2 \Gr \subset \bP. 
\end{equation}
By definition, {\GPK} threefolds are precisely the smooth $X_{g_1, g_2}$. 
Note that setting $H = \PGL(V)$, there is an embedding $H \to G$ induced by 
the isomorphism $\wedge^2V \cong W$.  

\begin{lemma}
\label{lemma-aut-group} 
Let $X = X_{g_1, g_2}$ be a {\GPK} threefold. 
The automorphism group scheme $\Aut(X)$ is finite and reduced, 
and can be described explicitly as 
\begin{equation*}
\Aut(X) = (g_1 H g_1^{-1} \cap g_2 H g_2^{-1}) \cup 
(g_2 H g_1^{-1} \cap g_1 H g_2^{-1}) 
\subset G. 
\end{equation*}
\end{lemma}

\begin{proof}
Since $X$ is a Calabi--Yau variety of Picard number $1$, 
it follows that $\Aut(X)$ is finite and reduced.  
Further, $\Aut(X)$ embeds in $G$ as the automorphisms $a \in G$ of $\bP$ which 
fix $X$, i.e. which satisfy 
\begin{equation*}
ag_1 \Gr \cap ag_2 \Gr = g_1 \Gr \cap g_2 \Gr. 
\end{equation*}
By Proposition~\ref{proposition-unique-Gr} this means 
either $ag_1 \Gr = g_1 \Gr$ and $ag_2 \Gr = g_2 \Gr$, or 
$ag_1 \Gr = g_2\Gr$ and $a g_2 \Gr = g_1 \Gr$. 
The first case is equivalent to $a \in g_1 H g_1^{-1} \cap g_2 H g_2^{-1}$ 
and the second to $a \in g_2 H g_1^{-1} \cap g_1 H g_2^{-1}$. 
\end{proof}

%%%%%%%%%%%%%%%%%%%%%%%%%%%%%%%%%%%%%%%%%%

\section{Moduli of {\GPK} threefolds} 
\label{section-moduli}

The goal of this section is to prove Theorem~\ref{theorem-moduli}. 
In \S\ref{subsection-moduli-stack-KCY} we construct the moduli stack $\cM$ 
of {\GPK} threefolds, and show it is a smooth separated Deligne--Mumford stack 
of finite type over $k$. 
In \S\ref{subsection-PGL-param} we construct the moduli stack $\cN$ of 
{\GPK} data (as a quotient of an open subspace of $\PGL(W) \times \PGL(W)$) 
and the $\PGL$-parameterization $f \colon \cN \to \cM$, and show that 
$\cN$ has the same properties as $\cM$. 
In \S\ref{subsection-PGL-derivative} we show that the derivative of $f$ at any 
point is an isomorphism. 
Finally, in \S\ref{subsection-theorem-moduli} we combine these results to prove 
Theorem~\ref{theorem-moduli}. 

\subsection{The moduli stack of {\GPK} threefolds}  
\label{subsection-moduli-stack-KCY} 
Let $P \in \bQ[t]$ be the Hilbert polynomial of a {\GPK} threefold.
Let $\Hilb$ denote the Hilbert scheme of closed subschemes of $\bP$ with Hilbert polynomial $P$, 
and let $\Hilb^{\circ} \subset \Hilb$ denote the open subscheme parameterizing Calabi--Yau threefolds 
which are smooth deformations of a {\GPK} threefold. 
The natural (left) action of $G = \PGL(W)$ on $\Hilb$ preserves $\Hilb^{\circ}$.  
Let 
\begin{equation}
\cM = [G \backslash \Hilb^{\circ}] 
\end{equation}
be the quotient stack and 
$q_{\cM} \colon \Hilb^\circ \to \cM$ the quotient morphism. 
We call  $\cM$ the \emph{moduli stack of {\GPK} threefolds} 
(although strictly speaking $\cM$ parameterizes smooth deformations of {\GPK} threefolds).  

\begin{lemma}
\label{lemma-M}
The stack $\cM$ is a smooth separated Deligne--Mumford stack of finite type over~$k$.  
Moreover, $\cM$ admits a coarse moduli space $\pi_{\cM} \colon \cM \to M$, where 
$M$ is a separated algebraic space of finite type over $k$. 
\end{lemma}

\begin{proof}
A geometric point of $\Hilb^{\circ}$ corresponds to a Calabi--Yau threefold 
of Picard number~$1$, so its stabilizer in $G$ is finite and reduced. 
Hence $\cM$ is Deligne--Mumford. 
The scheme $\Hilb^{\circ}$ is of finite type over $k$, and it is smooth by 
the Bogomolov--Tian--Todorov theorem on unobstructedness of Calabi--Yau varieties \cite{tian, todorov}. 
Hence $\cM$ is smooth and of finite type over $k$. 
Separatedness of $\cM$ follows from a result of 
Matsusaka and Mumford~\cite{matsusaka-mumford}.
Finally, the existence of a coarse 
moduli space with the stated properties then follows from a 
result of Keel and Mori~\cite{keel-mori}. 
\end{proof}

\begin{remark}
Let $X$ be a {\GPK} threefold. 
By Lemma~\ref{lemma-aut-group}, the automorphism group scheme 
$\Aut(X)$ coincides with the automorphism group scheme 
$\Aut_{\cM}([X])$ of the corresponding point $[X] \in \cM$. 
\end{remark}

\subsection{The $\PGL$-parameterization} 
\label{subsection-PGL-param} 

In \S\ref{subsection-aut-group} we observed any {\GPK} threefold can be 
written in the form~\eqref{Xg1g2}. 
We obtain a parameterization of $\cM$ by quotienting by the redundancies 
in this description, as follows. 
The quotient $G/H$ is a quasi-projective variety, which can be thought of as a parameter 
space for embeddings of the Grassmannian $\Gr(2,V)$ into $\bP$. 
Namely, for any point $g \in G$ the corresponding Grassmannian is $g \Gr$, 
which only depends on the coset~$gH$. 
Similarly $X_{g_1, g_2} \subset \bP$ only depends on the coset $(g_1H, g_2H)$, 
and the family of these varieties over $G \times G$ descends 
to a closed subscheme 
\begin{equation*}
\cX \subset G/H \times G/H \times \bP. 
\end{equation*}
Let $U \subset G/H \times G/H$ be the open subscheme parameterizing 
the smooth $3$-dimensional fibers of $\cX$. 
Then the restriction $\cX_U \to U$ is a family of {\GPK} threefolds, such that 
every {\GPK} threefold occurs as a fiber. 
This family induces a morphism $\tf \colon U \to \Hilb^\circ$. 

The group $\bZ/2 \times G$ acts on $G/H \times G/H$ (on the left), 
where $\bZ/2$ acts by swapping the two factors and $G$ acts by 
multiplication. This action preserves $U \subset G/H \times G/H$. 
Let 
\begin{equation}
\label{equation-N}
\cN = [(\bZ/2 \times G) \backslash U]
\end{equation} 
be the quotient stack and $q_{\cN} \colon U \to \cN$ the quotient morphism. 
We call $\cN$ the \emph{moduli stack of {\GPK} data}. 

The morphism $\tf \colon U \to \Hilb^\circ$ is equivariant with respect to the projection  
$\bZ/2 \times G \to G$, 
and hence descends to a morphism 
\begin{equation*}
f \colon \cN \to \cM,
\end{equation*}
which we call the \emph{$\PGL$-parameterization of $\cM$}. 

We note the following consequence of Proposition \ref{proposition-unique-Gr} 
and Lemma~\ref{lemma-aut-group}. 
Given a stack $\cY$ over~$k$, we denote by $|\cY(k)|$ the set of 
isomorphism classes of the $k$-points $\cY(k)$. 

\begin{lemma}
\label{lemma-f-points-auts}
We have: 
\begin{enumerate}
\item $f$ induces an injection $|\cN(k)| \to |\cM(k)|$. 
\item $f$ induces isomorphisms between the automorphism 
group schemes of points. 
\end{enumerate}
\end{lemma}

We have the following analog of Lemma~\ref{lemma-M} for $\cN$. 

\begin{lemma}
\label{lemma-N}
The stack $\cN$ is a smooth separated Deligne--Mumford stack of 
finite type over~$k$. Moreover, $\cN$ admits a coarse moduli space 
$\pi_{\cN} \colon \cN \to N$, where $N$ is a separated algebraic space of finite type over $k$.
\end{lemma}

\begin{proof}
Since the scheme $U$ is smooth and of finite type over $k$, so is the quotient stack $\cN$. 
Since $\cM$ is Deligne--Mumford, so is $\cN$ by Lemma~\ref{lemma-f-points-auts}. 
It remains to show that $\cN$ is separated; then the existence of a coarse 
moduli space with the stated properties follows from a result of 
Keel and Mori~\cite{keel-mori}. 
By the valuative criterion, this amounts to the following.  
Let $R$ be a valuation ring with field of fractions $K$, 
let $x, x' \colon \Spec(R) \to \cN$ be two $\Spec(R)$-points of $\cN$, 
and let $\gamma \colon x|_{K} \xrightarrow{\sim} x'|_{K}$ be 
an isomorphism of the restrictions to $\Spec(K)$. 
Then we must show there exists an isomorphism 
$\tilde{\gamma} \colon x \xrightarrow{\sim} x'$ restricting to $\gamma$. 
The points $x$ and $x'$ correspond to the data of  
embeddings $\Gr_{1,R}, \Gr_{2,R} \subset \bP_R$ and 
$\Gr'_{1,R}, \Gr'_{2,R} \subset \bP_R$ of $\Gr_R(2,V)$ such that 
\begin{equation*}
X_R = \Gr_{1, R} \cap \Gr_{2,R} \quad \text{ and } \quad 
X'_R = \Gr'_{1, R} \cap \Gr'_{2,R} 
\end{equation*}
are families of {\GPK} threefolds over $\Spec(R)$. 
Using the presentation $\bZ/2 = \set{\pm 1}$, 
the isomorphism $\gamma$ corresponds to a point 
$(\epsilon, a) \in \bZ/2 \times G(K)$ such that 
\begin{alignat*}{2}
a \Gr_{1,K} = \Gr'_{1,K} , ~~  &   a \Gr_{2,K} = \Gr'_{2,K} & \quad & \text{ if } \epsilon = +1, \\
a \Gr_{1,K} = \Gr'_{2,K} ,  ~~ &   a \Gr_{2,K} = \Gr'_{1,K} & \quad & \text{ if } \epsilon = -1.  
\end{alignat*}
In particular, $a$ is an automorphism of $\bP_K$ such that 
$a X_K = X'_K$. 
By separatedness of the moduli stack $\cM$ 
(Lemma~\ref{lemma-M}), 
we can find $\tilde{a} \in G(R)$ restricting to $a$ such 
that $\tilde{a} X_R = X'_R$.  
We claim 
\begin{alignat*}{2}
\ta \Gr_{1,R} = \Gr'_{1,R} , ~~  &   a \Gr_{2,R} = \Gr'_{2,R} & \quad & \text{ if } \epsilon = +1, \\
\ta \Gr_{1,R} = \Gr'_{2,R} ,  ~~ &   a \Gr_{2,R} = \Gr'_{1,R} & \quad & \text{ if } \epsilon = -1.  
\end{alignat*}
Indeed, by the separatedness of 
the parameter space $G/H$ for embeddings of 
$\Gr(2,V)$ into $\bP$, 
if two embeddings of $\Gr_R(2,V)$ into $\bP_R$ 
coincide after restriction to $K$, then they coincide.  
Hence $(\epsilon, \ta) \in \bZ/2 \times G(R)$ gives the desired 
extension of $\gamma$ to an isomorphism 
$\tilde{\gamma} \colon x \xrightarrow{\sim} x'$. 
\end{proof}

\subsection{Derivative of the $\PGL$-parameterization} 
\label{subsection-PGL-derivative}

Let $s \in U$ be a point, let $t = \tf(s) \in \Hilb^\circ$, and denote by 
$[s] \in \cN$ and $[t] \in \cM$ their images. 
Our goal is to prove that the derivative $\rd_{[s]} f \colon \rT_{[s]} \cN \to \rT_{[t]} \cM$ 
is an isomorphism. In fact, we will prove a slightly more precise result, which gives an 
explicit description of these tangent spaces. 

To formulate this, 
let $\act_s \colon G  \to U$ be the action morphism at $s$ given by $\act_s(g) = g \cdot s$,  
and similarly let $\act_t \colon G \to \Hilb^\circ$ be the action morphism at $t$ given by 
$\act_t(g) = g \cdot t$. 
Then there is a commutative diagram
\begin{equation}
\label{diagram-spaces}
\vcenter{
\xymatrix{
G \ar@{=}[d] \ar[rr]^{\act_s } & & U \ar[d]_-{\tf} \ar[rr]^{q_{\cN}} & & 
\cN \ar[d]_-{f} \\
G \ar[rr]^{\act_t} && \Hilb^\circ \ar[rr]^{q_{\cM}} & & \cM  
}
}
\end{equation}
Let $\fg$ denote the Lie algebra of $G$. 
The rest of this subsection is dedicated to the proof of the following result. 

\begin{proposition}
\label{proposition-derivative-f}
Taking derivatives in~\eqref{diagram-spaces} gives a 
commutative diagram
\begin{equation}
\label{tangent-spaces}
\vcenter{
\xymatrix{
0 \ar[r] & \fg \ar@{=}[d] \ar[rr]^{\rd_{1} \act_s } & & \rT_sU \ar[d]_{\rd_s \tf}^{\wr} \ar[rr]^{\rd_s q_{\cN}} & & 
\rT_{[s]} \cN \ar[d]_{\rd_{[s]} f}^{\wr} \ar[r] & 0 \\
0 \ar[r] & \fg \ar[rr]^{\rd_{1} \act_t} && \rT_t \Hilb^\circ \ar[rr]^{\rd_t q_{\cM}} & & \rT_{[t]} \cM \ar[r] & 0  
}
}
\end{equation}
with exact rows and vertical maps isomorphisms. 
\end{proposition}

From the presentations of $\cN$ and $\cM$ as quotient stacks it 
follows that the rows of~\eqref{tangent-spaces} are right exact, 
but since $\cN$ and $\cM$ are Deligne--Mumford (Lemmas~\ref{lemma-N} 
and~\ref{lemma-M}), they are in fact exact.  
Hence to prove Proposition~\ref{proposition-derivative-f}, 
it suffices to show $\rd_{s}\tf \colon \rT_sU \to \rT_t \Hilb^\circ$ is an isomorphism. 

To this end, we factor $\tf \colon U \to \Hilb^\circ$ as follows. 
Let $Q \in \bQ[t]$ be the Hilbert polynomial of $\Gr \subset \bP$, and let 
$\Hilb_Q$ be the Hilbert scheme of closed subschemes of $\bP$ with Hilbert 
polynomial $Q$. 
Let $\cY \subset \Hilb_Q \times \bP$ denote the universal family, and 
let $\cY_i \to \Hilb_Q \times \Hilb_Q$, $i=1,2,$ denote its pullback along each of the projections. 
Define $U' \subset \Hilb_Q \times \Hilb_Q$ to be the open subscheme   
over which the fibers of the morphism 
\begin{equation*}
\cY_1 \times_{\bP} \cY_2 \to \Hilb_Q \times \Hilb_Q 
\end{equation*}
are smooth deformations of a {\GPK} threefold, and let 
\begin{equation*}
\tf' \colon U' \to \Hilb^\circ 
\end{equation*}
be the induced morphism. 
The morphism $\tf$ factors through $\tf'$. 
Indeed, consider the closed subscheme $\cZ \subset G/H \times \bP$ 
whose fiber over $[g] \in G/H$ is $g\Gr \subset \bP$. 
This induces a morphism 
\begin{equation*}
\gamma \colon G/H \to \Hilb_Q 
\end{equation*} 
such that $\gamma \times \gamma \colon G/H \times G/H \to \Hilb_Q \times \Hilb_Q$ 
takes $U$ into $U'$, and hence induces a morphism 
\begin{equation*}
j \colon U \to U'. 
\end{equation*}
It follows from the definitions that $\tf = \tf' \circ j$. 
Thus $\rd_s\tf = \rd_{j(s)} \tf' \circ \rd_{s} j$, 
so to prove Proposition~\ref{proposition-derivative-f} it suffices to show 
$\rd_{s}j$ and $\rd_{j(s)}\tf'$ are isomorphisms. 
This is the content of the next two lemmas. 

\begin{lemma}
\label{lemma-dj}
The map $\rd_{s}j \colon \rT_{s}U \to \rT_{j(s)}U'$ is an isomorphism. 
\end{lemma}

\begin{proof}
By the definition of $j$, it suffices to show that for 
any $g \in G$ the map 
\begin{equation*}
\rd_{gH} {\gamma} \colon \rT_{gH}(G/H) \to \rT_{\gamma(gH)} \Hilb_Q
\end{equation*}
is an isomorphism. 
The point $\gamma(gH) \in \Hilb_Q$ corresponds to the 
subscheme $Z = g \Gr \subset \bP$, and there is a canonical 
isomorphism $\rT_{\gamma(gH)} \Hilb_Q \cong \rH^0(Z, \rN_{Z/\bP})$. 
The exact sequence 
\begin{equation*}
0 \to \rT_Z \to \rT_{\bP} \vert_Z \to \rN_{Z/\bP} \to 0 
\end{equation*}
induces a long exact sequence 
\begin{equation*}
0 \to \rH^0(Z, \rT_Z) \to \rH^0(Z, \rT_{\bP} \vert_Z) \to 
\rH^0(Z, \rN_{Z/\bP}) \to \rH^1(Z, \rT_Z) \to \dots 
\end{equation*}
We have $\rH^1(Z, \rT_Z) = 0$, 
so the first three terms form a short exact sequence.  
Moreover, there is a canonical isomorphism $\fh \cong \rH^0(Z, \rT_Z)$, 
since $\rH^0(Z, \rT_Z)$ is identified with the tangent space to 
$H \cong \Aut(Z)$. 
By the same reason $\fg \cong \rH^0(\bP, \rT_{\bP})$. 
By Lemma~\ref{lemma-restriction-TPZ} the restriction map 
$\rH^0(\bP, \rT_{\bP}) \to \rH^0(Z, \rT_{\bP} \vert_Z)$ is an isomorphism, 
so we get an isomorphism $\fg \cong \rH^0(\bP, \rT_{\bP}\vert_Z)$. 
The isomorphisms $\fh \cong  \rH^0(Z, \rT_Z)$ 
and $\fg \cong \rH^0(Z, \rT_{\bP} \vert_Z)$ are compatible with the canonical isomorphism 
$\rT_{gH}(G/H) \cong \fg/\fh$, i.e. they fit into a commutative diagram  
\begin{equation*}
\xymatrix{
0 \ar[r] & \fh \ar[r] \ar[d]^-{\wr} & \fg \ar[r] \ar[d]^-{\wr} & \rT_{gH}(G/H) \ar[r] \ar[d]^{\rd_{gH} \gamma}& 0 \\
0 \ar[r] & \rH^0(Z, \rT_Z) \ar[r] & \rH^0(Z, \rT_{\bP} \vert_Z) \ar[r] & \rH^0(Z, \rN_{Z/\bP}) \ar[r] & 0
}
\end{equation*}
So $\rd_{gH} \gamma$ is an isomorphism. 
\end{proof}

\begin{lemma}
\label{lemma-derivative-j}
The map $\rd_{j(s)} \tf' \colon \rT_{j(s)} U' \to \rT_{t} \Hilb^\circ$ is an isomorphism. 
\end{lemma}

\begin{proof}
Write $s = (g_1H, g_2H)$ and let $\Gr_i = g_i\Gr$, so that 
$X = \Gr_1 \cap \Gr_2$ is the {\GPK} threefold corresponding to $s$. 
Let $\rN_i = \rN_{\Gr_i/\bP}$.  
Then since $U' \subset \Hilb_Q \times \Hilb_Q$ is an open subscheme, there is a 
canonical isomorphism 
\begin{equation*}
\rT_{j(s)} U' \cong \rH^0(\Gr_1, \rN_{1}) \oplus \rH^0(\Gr_2, \rN_2). 
\end{equation*}
Similarly, since $\rN_{X/\bP} \cong \rN_{1}\vert_X \oplus \rN_2 \vert_X$, 
there is an isomorphism 
\begin{equation*}
\rT_{t} \Hilb^\circ \cong \rH^0(X, \rN_1\vert_X) \oplus \rH^0(X, \rN_2\vert_X) . 
\end{equation*}
Under the above isomorphisms, the map $\rd_{j(s)} \tf' \colon \rT_{j(s)} U' \to \rT_{t} \Hilb^\circ$ 
is identified with the direct sum of the restriction maps 
$\rH^0(\Gr_i, \rN_i) \to \rH^0(X, \rN_i \vert_X)$. 
Now use Lemma~\ref{lemma-restriction-Ni}. 
\end{proof}

\subsection{Proof of Theorem~\ref{theorem-moduli}} 
\label{subsection-theorem-moduli} 
We have already shown $\cN$ and $\cM$ are 
smooth separated Deligne--Mumford stacks of finite type over $k$ 
(Lemmas~\ref{lemma-N} and~\ref{lemma-M}), so we just need 
to show $f$ is an open immersion. 
The separatedness of $\cN$ guarantees that $f$ is separated. 
By Proposition~\ref{proposition-derivative-f} and the smoothness 
of $\cN$ and $\cM$, the morphism $f$ is \'{e}tale. 
Now the result follows by replacing $\cM$ with the image 
of $f$ and applying \cite[\href{http://stacks.math.columbia.edu/tag/0DUD}{Tag 0DUD}]{stacks-project}, 
whose hypotheses hold by the above observations and Lemma~\ref{lemma-f-points-auts}. 

%%%%%%%%%%%%%%%%%%%%%%%%%%%%%%%%%%%%%%%%%

\section{The infinitesimal structure of $\cN$}
\label{section-infinitesimal-computations} 

In this section, we study the moduli stack $\cN$ infinitesimally.  
Our goal is to prove Theorem~\ref{theorem-faithful-action} from the introduction, 
as well as the following result. 

\begin{proposition} 
\label{proposition-tr-involution}
Let $s \in \cN$ be a point. 
Let $\gamma \in \Aut_{\cN}(s)$ be an element such that the induced map  
$\gamma_* \in \GL(\rT_{s} \cN)$ is an involution. 
Then 
\begin{equation*}
\tr(\gamma_*) \in \set{51, 3, 1, -3, -5, -13, -15, -35} . 
\end{equation*}
\end{proposition}

We start in \S\ref{subsection-tangent-space-N} by spelling out an explicit 
presentation of the tangent space to any point \mbox{$s \in \cN$}. 
In \S\ref{subsection-aut-action}, 
we combine this with our description of $\Aut_{\cN}(s)$ 
from Lemma~\ref{lemma-aut-group} to prove some preliminary results, 
by analyzing the eigenvalues of the induced maps on $\rT_{s} \cN$. 
As an easy consequence of our analysis, we prove Theorem~\ref{theorem-faithful-action} 
and Proposition~\ref{proposition-tr-involution}. 

\subsection{An explicit presentation of the tangent space to $\cN$}
\label{subsection-tangent-space-N}
Proposition~\ref{proposition-derivative-f} gives a presentation of the 
tangent spaces to $\cN$. 
To make this explicit, we start 
with some preliminary remarks. 

For any $g \in G$ there is an isomorphism 
\begin{equation}
\label{TGH}
\fg/\fh \cong \rT_{gH}(G/H) . 
\end{equation}
If we regard $\rT_{gH}(G/H)$ as the set of $k[\vepsilon]/(\vepsilon^2)$-points of $G/H$ 
based at $gH$, 
then this identification is induced by the map  
\begin{align*}
\eta_g \colon \fg & \to \rT_{gH}(G/H) \\ 
R & \mapsto g(1 + \vepsilon R)H . 
\end{align*}
As the notation indicates, the identification~\eqref{TGH} depends on 
the choice of representative for the coset $gH \in G/H$.  
Namely, suppose $gH = g'H$. 
Then there is a commutative diagram 
\begin{equation*}
\xymatrix{
\fg \ar[r]^{\Ad_{(g')^{-1}g}} \ar[d]_{\eta_g} & \fg \ar[d]^{\eta_{g'}} \\ 
\rT_{gH}(G/H) \ar@{=}[r] & \rT_{g'H}(G/H)
}
\end{equation*}
Here, for any $a \in G$ we use the notation $\Ad_{a} \colon \fg \to \fg$ for 
the action of $a$ under the adjoint representation, i.e. $\Ad_a(R) = aRa^{-1}$.  
This follows from the computation  
\begin{equation*}
g(1+ \vepsilon R)H = (1+ \vepsilon gRg^{-1})gH = (1+\vepsilon gRg^{-1})g'H 
= g'(1+\vepsilon (g')^{-1}gRg^{-1}g')H. 
\end{equation*}

Next we note that the group $G$ acts on $G/H$ on the left. 
For $gH \in G/H$ the derivative at $1 \in G$ of the action morphism 
$\act_{gH} \colon G \to G/H$, $\act_{gH}(a) = agH$, gives a map 
\begin{equation*}
\fg \to \rT_{gH}(G/H). 
\end{equation*}
Under the identification~\eqref{TGH}, this map takes the form 
\begin{align*}
\fg & \to \fg/\fh  \\
S & \mapsto g^{-1} S g. 
\end{align*}
Indeed, this follows from the observation  
\begin{equation*}
(1+\vepsilon S)gH = g(1 + \vepsilon g^{-1} Sg)H. 
\end{equation*}

Combined with Proposition~\ref{proposition-derivative-f}, the above discussion 
gives the following. 

\begin{lemma}
\label{lemma-presentation-TN}
Let $(g_1, g_2) \in G \times G$ be such that $[(g_1H, g_2H)] \in \cN$. 
\begin{enumerate}
\item There is a short exact sequence 
\begin{equation}
\label{Tg1Hg2H}
0 \to \fg  \to
\fg/\fh \oplus \fg/\fh  \to \rT_{[(g_1H, g_2H)]}\cN \to 0  
\end{equation}
where the first map is given by 
\begin{equation}
\begin{aligned}
\label{S-map}
\fg  & \to \fg/\fh \oplus \fg/\fh \\
S    & \mapsto  (g_1^{-1}Sg_1, g_2^{-1}Sg_2) . 
\end{aligned}
\end{equation}

\item 
\label{presentation-TN-functorial}
The sequence~\eqref{Tg1Hg2H} depends on the choice 
of representatives $g_1, g_2$ for the cosets $g_1H, g_2H$, but 
is canonical in the following sense. If $g_1H = g'_1H$ and $g_2 H = g'_2H$, 
then there is a commutative diagram of exact sequences 
\begin{equation*}
\xymatrix{
0 \ar[r] & \fg  \ar[rr] \ar[d] && \fg/\fh \oplus \fg/\fh  \ar[rr] \ar[d]_{(\Ad_{(g'_1)^{-1}g_1} , }^{\Ad_{(g'_2)^{-1}g_2})} && \rT_{[(g_1H, g_2H)]}\cN \ar[r] \ar@{=}[d] & 0 \\
0 \ar[r] & \fg  \ar[rr] && \fg/\fh \oplus \fg/\fh  \ar[rr] && \rT_{[(g'_1H, g'_2H)]}\cN \ar[r] & 0
}
\end{equation*}
\end{enumerate}
\end{lemma}

In particular, suppose 
$[(g_1H, g_2H)], [(g'_1H, g'_2H)] \in \cN$. 
Then to specify a map 
\begin{equation*}
\alpha \colon \rT_{[(g_1H, g_2H)]}\cN \to \rT_{[(g'_1H, g'_2H)]}\cN 
\end{equation*} 
it suffices to specify a map 
\begin{equation*}
\beta \colon \fg \oplus \fg \to \fg \oplus \fg
\end{equation*}
which preserves the subspace 
$\fh \oplus \fh \subset \fg \oplus \fg$ and the image of the map $\fg \to \fg/\fh \oplus \fg/\fh$ in~\eqref{S-map}. 
To indicate this situation, 
we will simply say \emph{$\alpha$ is induced by the map $\beta$}. 
We emphasize once again that this notion depends on the choice of 
representatives $g_1,g_2,g'_2, g'_2$, which we regard as being 
made implicitly in the notation for the domain and target of $\alpha$. 

\subsection{Eigenvalue analysis}
\label{subsection-aut-action}  

Recall that $\cN$ is defined as the quotient of an open subscheme $U \subset G/H \times G/H$ 
by the group $\bZ/2 \times G$. 
Hence, letting $\sigma \in \bZ/2$ denote the generator,  
for any point $[(g_1H, g_2H)] \in \cN$ and $a \in G$ there are corresponding isomorphisms 
\begin{align*}
\gamma_{(1, a)} & \colon [(g_1H, g_2H)] \xrightarrow{\sim} [(ag_1H, ag_2H)], \\
\gamma_{(\sigma, a)} & \colon [(g_1H, g_2H)] \xrightarrow{\sim} [(ag_2H, ag_1H)]. 
\end{align*} 
These isomorphisms of points of $\cN$ induce isomorphisms of tangent spaces, which 
we denote respectively by 
\begin{align*}
a_*  & \colon \rT_{[(g_1H, g_2H)]}\cN \to \rT_{[(ag_1H, ag_2H)]}\cN , \\
(\sigma,a)_* & \colon \rT_{[(g_1H, g_2H)]}\cN \to \rT_{[(ag_2H, ag_1H)]}\cN .
\end{align*}
We also simply write $\sigma_*$ for $(\sigma, 1)_*$. 
The next lemma follows immediately by unwinding the definitions. 

\begin{lemma}
\label{lemma-a-sigma-action}
Let $(g_1, g_2) \in G \times G$ be such that $[(g_1H, g_2H)] \in \cN$. 
\begin{enumerate}
\item 
\label{lemma-a-action}
The map  
\begin{equation*}
a_*   \colon \rT_{[(g_1H, g_2H)]}\cN \to \rT_{[(ag_1H, ag_2H)]}\cN
\end{equation*}
is induced by the identity map 
\begin{equation*}
\id \colon \fg \oplus \fg  \to \fg \oplus \fg. 
\end{equation*}

\item 
\label{lemma-sigma-action}
The map 
\begin{equation*}
\sigma_* \colon \rT_{[(g_1H, g_2H)]}\cN \to \rT_{[(g_2H, g_1H)]}\cN 
\end{equation*}
is induced by the transposition map 
\begin{align*}
\fg \oplus \fg & \to \fg \oplus \fg \\
(R_1, R_2) & \mapsto (R_2, R_1) .  
\end{align*}
\end{enumerate}
\end{lemma}

Below we will be concerned with automorphisms of a point $[(g_1H, g_2H)] \in \cN$. 
As observed in Lemma~\ref{lemma-f-points-auts}, the automorphism group of 
$[(g_1H, g_2H)] \in \cN$ coincides with that of $[X_{g_1,g_2}] \in \cM$. 
This latter group consists of elements of two types, according to Lemma~\ref{lemma-aut-group}. 
We turn this into a definition: 

\begin{definition}
\label{definition-aut-types}
Let $(g_1, g_2) \in G \times G$ be such that $s = [(g_1H, g_2H)] \in \cN$, 
i.e. such that $X_{g_1, g_2}$ is a {\GPK} threefold. 
We say an automorphism $a \in G$ of $X_{g_1, g_2}$ is: 
\begin{enumerate}
\item of \emph{type I} 
if $a \in g_1 H g_1^{-1} \cap g_2 H g_2^{-1}$; 

\item of \emph{type II} 
if $a \in g_2 H g_1^{-1} \cap g_1 H g_2^{-1}$. 
\end{enumerate}
We say $\gamma \in \Aut_{\cN}(s)$ is of \emph{type I} or \emph{type II} 
according to the type of the corresponding element of $\Aut_{\cM}([X_{g_1,g_2}])$ 
under the isomorphism $\Aut_{\cN}(s) \cong \Aut_{\cM}([X_{g_1,g_2}])$. 
\end{definition}

In the situation of Definition~\ref{definition-aut-types}, 
the automorphism of $[(g_1H, g_2H)]$ corresponding to $a$ 
is $\gamma_{(1,a)}$ from above if $a$ is of type I, 
and is $\gamma_{(\sigma, a)}$ if $a$ is of type II. 
Via the isomorphism $\rT_{[(g_1H, g_2H)]}\cN \cong \rT_{[X_{g_1,g_2}]} \cM$
from Proposition~\ref{proposition-derivative-f}, the induced 
map 
\begin{equation*}
a_*   \colon \rT_{[X_{g_1,g_2}]} \cM \to \rT_{[X_{g_1,g_2}]}  \cM
\end{equation*}
is identified with 
\begin{alignat*}{2}
a_*  & \colon \rT_{[(g_1H, g_2H)]}\cN \to \rT_{[(g_1H, g_2H)]}\cN & \quad & \text{ if $a$ is of type I} , \\
(\sigma,a)_* & \colon \rT_{[(g_1H, g_2H)]}\cN \to \rT_{[(g_1H, g_2H)]}\cN & \quad & \text{ if $a$ is of type II} . 
\end{alignat*}
In the following lemma, we describe these maps explicitly. 

\begin{lemma}
\label{lemma-aut-action}
Let $(g_1, g_2) \in G \times G$ be such that $[(g_1H, g_2H)] \in \cN$. 
Let $a \in G$ be an automorphism of $X_{g_1,g_2}$. 
\begin{enumerate}
\item 
\label{type-I-action}
If $a$ is of type I, then the map 
\begin{equation*}
a_* \colon \rT_{[(g_1H, g_2H)]}\cN \to \rT_{[(g_1H, g_2H)]}\cN 
\end{equation*} 
is induced by the map 
\begin{align*}
\fg \oplus \fg & \to \fg \oplus \fg \\
(R_1, R_2) & \mapsto 
((g_1^{-1}ag_1)R_1(g_1^{-1}ag_1)^{-1}, 
(g_2^{-1}ag_2)R_2(g_2^{-1}ag_2)^{-1}) . 
\end{align*} 

\item 
\label{type-II-action}
If $a$ is of type II, then the map 
\begin{equation*}
(\sigma,a)_* 
\colon \rT_{[(g_1H, g_2H)]}\cN \to \rT_{[(g_1H, g_2H)]}\cN 
\end{equation*} 
is induced by the map 
\begin{align*}
\fg \oplus \fg & \to \fg \oplus \fg \\
(R_1, R_2) & \mapsto 
((g_1^{-1}ag_2)R_2(g_1^{-1}ag_2)^{-1}, 
(g_2^{-1}ag_1)R_1(g_2^{-1}ag_1)^{-1}) . 
\end{align*} 
\end{enumerate}
\end{lemma}

\begin{proof}
By Lemma~\ref{lemma-a-sigma-action}\eqref{lemma-a-action} the map 
\begin{equation*}
a_* \colon \rT_{[(g_1H, g_2H)]}\cN \to \rT_{[(ag_1H, ag_2H)]}\cN 
\end{equation*}
is induced by the identity $\id \colon \fg \oplus \fg \to \fg \oplus \fg$. 
If $a$ is of type I, then we have $ag_1H = g_1H$ and $ag_2H = g_2H$. 
Hence~\eqref{type-I-action} follows from the comparison between the 
presentations for $\rT_{[(ag_1H, ag_2H)]}$ and $\rT_{[(g_1H, g_2H)]}$  
given by Lemma~\ref{lemma-presentation-TN}\eqref{presentation-TN-functorial}. 
Part~\eqref{type-II-action} follows similarly. 
\end{proof}

Set $\tH = \GL(V)$, $\tG = \GL(W)$, $\tfh = \gl(V)$, and $\tfg = \gl(W)$, 
so that we have commutative diagrams 
\begin{equation}
\vcenter{
\label{HGhg}
\xymatrix{
\tH \ar[r] \ar[d] & \tG  \ar[d]  & \quad &  \tfh \ar[r] \ar[d] & \tfg \ar[d]  \\
H \ar[r] & G & \quad & \fh \ar[r] & \fg 
}
}
\end{equation}
where the rows are embeddings and the columns surjections with $1$-dimensional kernels. 
In the following proofs, at certain points it will be convenient to work with the 
spaces in the top row. 
For readability, we commit the following abuse of notation: given $g \in G$ we  
use the same symbol to denote a fixed lift $g \in \tG$; similarly for $H$ and $\tH$; 
and we choose our lifts compatibly, i.e. if $g \in G$ 
is the image of $g_0 \in H$ under $H \to G$, we choose a lift of $g_0$ that maps 
to $g$. 

\begin{lemma}
\label{lemma-aut-action-CY}
Let $(g_1, g_2) \in G \times G$ be such that $[(g_1H, g_2H)] \in \cN$. 
Let $a \in G$ be an automorphism of $X_{g_1,g_2}$. 
\begin{enumerate}
\item \label{aut-action-I}
If $a$ is of type I and the map 
\begin{equation*}
a_* \colon \rT_{[(g_1H, g_2H)]}\cN \to \rT_{[(g_1H, g_2H)]}\cN 
\end{equation*} 
is the identity, then $a = 1$. 

\item \label{aut-action-II}
If $a$ is of type II and the map 
\begin{equation*}
(\sigma,a)_* \colon \rT_{[(g_1H, g_2H)]}\cN \to \rT_{[(g_1H, g_2H)]}\cN 
\end{equation*} 
is an involution, then $a^2 = 1$. 
\end{enumerate}
\end{lemma}

\begin{proof}
Note that there is an isomorphism $[(g_1H, g_2H)] \cong [(H, g_1^{-1}g_2H)] \in \cN$.  
Hence we may assume $g_1 = 1$ and $g_2 = g$. 
Further, note that $a \in G$ must have finite order by Lemma~\ref{lemma-aut-group}, 
so in particular $a$ is diagonalizable. 
Since the square of an automorphism of type II is of type I, 
\eqref{aut-action-II} follows from \eqref{aut-action-I}. 

So assume $a$ is of type I, i.e. $a \in H \cap gHg^{-1} \subset G$. 
Then by Lemma~\ref{lemma-aut-action}\eqref{type-I-action} the map 
$a_* \colon \rT_{[(H, gH)]}\cN \to \rT_{[(H, gH)]}\cN$ 
is induced by the map
\begin{equation} 
\label{a-I-action}
\begin{aligned}
\fg \oplus \fg & \to \fg \oplus \fg \\
(R_1, R_2) & \mapsto 
(aR_1a^{-1}, (g^{-1}ag)R_2(g^{-1}ag)^{-1}) . 
\end{aligned}
\end{equation} 
Recall that this means that in terms of the presentation 
\begin{equation}
\label{a-I-presentation}
\begin{array}{ccccccccc}
0 & \to &  \fg & \to &  \fg/\fh \oplus \fg/\fh & \to & \rT_{[(H, gH)]}\cN & \to &  0 \\ 
&& S & \mapsto &   (S, g^{-1}Sg) &  & & &
\end{array}
\end{equation}
given by \eqref{Tg1Hg2H}, $a_*$ is induced by \eqref{a-I-action}. 

Let $a_0 \in H$ be the element whose image under the embedding $H \to G$ is $a$. 
Let $\lambda_i, 1 \leq i \leq 5$, be the eigenvalues of $a_0$. 
Then $\lambda_i\lambda_j, 1 \leq i < j \leq 5$, are the eigenvalues of $a$; 
we must show they are all equal if $a_*$ is the identity. 
We start by computing the eigenvalues of $a_*$ in terms of the $\lambda_i$. 
We do so by computing the eigenvalues on each summand in $\fg/\fh \oplus \fg/\fh$ 
and on the subspace $\fg \subset \fg/\fh \oplus \fg/\fh$ separately: 

\medskip \noindent
\textit{The first $\fg/\fh$ summand}. Consider the map $\tfg \to \tfg$ given by $R \mapsto aRa^{-1}$. It  
has eigenvalues 
\begin{equation*}
\frac{\lambda_i\lambda_j}{\lambda_k \lambda_{\ell}}, ~ 1 \leq i < j \leq 5, \, 1 \leq k < \ell \leq 5, 
\end{equation*} 
and similarly the induced map $\tfh \to \tfh$ has eigenvalues 
\begin{equation*}
\frac{\lambda_i}{\lambda_k} , ~ 1 \leq i,k \leq 5.  
\end{equation*}
Hence the induced map $\tfg/\tfh \to \tfg/\tfh$, which coincides with the induced 
map $\fg/\fh \to \fg/\fh$, has eigenvalues given by the multiset difference 
\begin{equation}
\label{multiset-1}
\set{ \frac{\lambda_i\lambda_j}{\lambda_k \lambda_{\ell}} \st 1 \leq i < j \leq 5, \, 1 \leq k < \ell \leq 5} 
- 
\set{\frac{\lambda_i}{\lambda_k} \st 1 \leq i,k \leq 5} . 
\end{equation}

\medskip \noindent
\textit{The second $\fg/\fh$ summand}.
Let $b = g^{-1}ag \in G$. By assumption $b$ is in the image of the embedding $H \to G$; 
let $b_0$ be its preimage. 
Let $\mu_i, 1 \leq i \leq 5$, be the eigenvalues of $b_0$, so that 
$\mu_i\mu_j, 1 \leq i < j \leq 5$, are the eigenvalues of $b$. 
Note that we have an equality of multisets 
\begin{equation}
\label{multiset-2}
\set{ \mu_i\mu_j \st 1 \leq i < j \leq 5 } = 
\set{ \lambda_i\lambda_j \st 1 \leq i < j \leq 5} , 
\end{equation}
but the multisets $\set{\mu_i \st1 \leq i \leq 5}$ and 
$\set{\lambda_i \st 1 \leq i \leq 5}$ need not coincide. 
The above argument shows that the map $\fg/\fh \to \fg/\fh$ induced 
by $R \mapsto bRb^{-1}$ has eigenvalues given by the multiset 
\begin{equation}
\label{multiset-3}
\set{ \frac{\mu_i\mu_j}{\mu_k \mu_{\ell}} \st 1 \leq i < j \leq 5, \, 1 \leq k < \ell \leq 5} 
- 
\set{\frac{\mu_i}{\mu_k} \st 1 \leq i,k \leq 5} . 
\end{equation}

\medskip \noindent
\textit{The subspace $\fg \subset \fg/\fh \oplus \fg/\fh$}.
The map~\eqref{a-I-action} induces the 
map $S \mapsto a S a^{-1}$ on the copy of $\fg \subset \fg/\fh \oplus \fg/\fh$ 
embedded as in~\eqref{a-I-presentation}. 
The above argument shows this map has eigenvalues given by the multiset 
\begin{equation}
\label{multiset-4}
\set{ \frac{\lambda_i\lambda_j}{\lambda_k \lambda_{\ell}} \st 1 \leq i < j \leq 5, \, 1 \leq k < \ell \leq 5 } - 
\set{1} ,
\end{equation}
where we have removed a single $1$ eigenvalue corresponding to the kernel of $\tfg \to \fg$. 

\medskip \noindent
\textit{The eigenvalues of $a_*$}. 
Combining all of the above, we conclude that the eigenvalues 
of $a_*$ are given by the multiset sum of~\eqref{multiset-1} 
and~\eqref{multiset-3} minus~\eqref{multiset-4}, i.e. 
by 
\begin{equation}
\label{multiset-eigenvalues}
\set{1} + 
\set{ \frac{\mu_i\mu_j}{\mu_k \mu_{\ell}} \st 1 \leq i < j \leq 5, \, 1 \leq k < \ell \leq 5} 
- 
\set{\frac{\lambda_i}{\lambda_k} \st 1 \leq i,k \leq 5}
- 
\set{\frac{\mu_i}{\mu_k} \st 1 \leq i,k \leq 5} .  
\end{equation}

\medskip 
Recall that to finish we need to show that if the support of~\eqref{multiset-eigenvalues} is $\set{1}$, 
then the $\lambda_i$ coincide. 
To see this, first note that every $\lambda_i/\lambda_k$ appears at 
least three times in 
\begin{equation}
\label{multiset-5}
\set{ \frac{\lambda_i\lambda_j}{\lambda_k \lambda_{\ell}} \st 1 \leq i < j \leq 5, \, 1 \leq k < \ell \leq 5} , 
\end{equation}
hence the difference~\eqref{multiset-1} has the same support as~\eqref{multiset-5}. 
Similarly, the multiset~\eqref{multiset-3} has the same support as
\begin{equation}
\label{multiset-6}
\set{ \frac{\mu_i\mu_j}{\mu_k \mu_{\ell}} \st 1 \leq i < j \leq 5, \, 1 \leq k < \ell \leq 5}. 
\end{equation}
But using~\eqref{multiset-2} we see that twice the multiset~\eqref{multiset-eigenvalues} 
coincides with the sum of $\set{1,1}$, \eqref{multiset-1}, and \eqref{multiset-3}. 
It follows that if the support of~\eqref{multiset-eigenvalues} is $\set{1}$, then 
so are the supports of \eqref{multiset-5} and \eqref{multiset-6}, and hence 
all $\lambda_i$ coincide. 
\end{proof}

In the next proof, we use the following convenient notation. 
Given an endomorphism $\psi$ of a $k$-vector space and $\lambda \in k$, 
we write $\mult_{\lambda}(\psi)$ for the multiplicity of the eigenvalue $\lambda$ for~$\psi$.   
If $\psi$ is an involution, its eigenvalues are $\pm1$, and we say:  
$\psi$ is of type $(p,q)$ if $\mult_{1}(\psi) = p$ and $\mult_{-1}(\psi) = q$;  
$\psi$ is of type $\set{p,q}$ if it is either of type $(p,q)$ or $(q,p)$.
Keep in mind below our abuse of notation by which given $a$ in 
$G$ or $H$, we fix a lift to $\tG$ or $\tH$ denoted by the same symbol; 
if $a$ is an involution, we choose our lift to be an involution as well. 

\begin{lemma}
\label{lemma-aut-involution}
Let $(g_1, g_2) \in G \times G$ be such that $[(g_1H, g_2H)] \in \cN$. 
Let $1 \neq a \in G$ be an automorphism of $X_{g_1,g_2}$ which satisfies $a^2 = 1$.  
\begin{enumerate}
\item \label{aut-involution-I}
If $a$ is of type I, then the trace of the map  
\begin{equation*}
a_* \colon \rT_{[(g_1H, g_2H)]}\cN \to \rT_{[(g_1H, g_2H)]}\cN 
\end{equation*} 
is one of the following: $3, -5, -13$. 

\item \label{aut-involution-II}
If $a$ is of type II, then the trace of the map 
\begin{equation*}
(\sigma,a)_* \colon \rT_{[(g_1H, g_2H)]}\cN \to \rT_{[(g_1H, g_2H)]}\cN 
\end{equation*} 
is one of the following: $1, -3, -15, -35$. 
\end{enumerate}
\end{lemma}

\begin{proof}
As in the proof of Lemma~\ref{lemma-aut-action-CY}, we may assume 
$g_1 = 1$ and $g_2 = g$. 

Assume $a$ is of type I. 
To compute the trace of $a_*$, it suffices to compute $\mult_1(a_*)$. 
For this, we follow the proof of Lemma~\ref{lemma-aut-action-CY}. 
As there, let $a_0 \in H$ be a preimage of $a$, and let $b_0 \in H$ be a 
preimage of $b = g^{-1}ag$. 
Then the formula~\eqref{multiset-eigenvalues} for the multiset of eigenvalues 
of $a_*$ shows 
\begin{equation*}
\mult_1(a_*) = 1 + \mult_1 (\Ad_b \colon \tfg/\tfh \to \tfg/\tfh ) -  
\mult_1 (\Ad_{a_0} \colon \tfh \to \tfh ) -
\mult_1 (\Ad_{b_0} \colon \tfh \to \tfh ). 
\end{equation*}

To compute the above quantity, we use the following remark. 
Let $\psi \colon L \to L$ be an involution of a $k$-vector space $L$ of type $\set{p,q}$.   
Then the $+1$ eigenspace of the map 
\begin{align*}
\Ad_{\psi} \colon & \gl(L) \to \gl(L) \\ 
& R  \mapsto  \psi R \psi^{-1}
\end{align*}
consists of $R \in \gl(L)$ that commute with $\psi$, and hence $\mult_{1}(\Ad_{\psi}) = p^2 + q^2$. 
Since $a_0 \neq 1$ by assumption, 
the involution $a_0$ is either of type $\set{4,1}$ or $\set{3,2}$, and $a$ is of type $\set{4,6}$. 
Similarly, $b_0$ is either of type $\set{4,1}$ or $\set{3,2}$, and $b$ is of type $\set{4,6}$. 
Thus, using the above formula we find  
\begin{equation*}
\mult_1(a_*) = 
\begin{cases}
19 & \text{ if $a_0$ and $b_0$ are of type $\set{4,1}$}, \\
23 & \text{ if $a_0$ and $b_0$ are of different $\set{p,q}$ types}, \\
27 & \text{ if $a_0$ and $b_0$ are of type $\set{3,2}$}. 
\end{cases}
\end{equation*}
Since $\dim \rT_{[(H, gH)]}\cN = 51$, this gives for the trace 
\begin{equation*}
\tr(a_*) =
\begin{cases}
-13 & \text{ if $a_0$ and $b_0$ are of type $\set{4,1}$}, \\
-5 & \text{ if $a_0$ and $b_0$ are of different $\set{p,q}$ types} , \\
3 & \text{ if $a_0$ and $b_0$ are of type $\set{3,2}$}, 
\end{cases}
\end{equation*}
and hence proves \eqref{aut-involution-I} of the lemma. 

Now assume $a$ is of type II, i.e. $a \in gH \cap Hg^{-1} \subset G$. 
Then by Lemma~\ref{lemma-aut-action}\eqref{type-II-action} the  
map 
$(\sigma,a)_* \colon \rT_{[(H, gH)]}\cN \to \rT_{[(H, gH)]}\cN$  
is induced by the map 
\begin{equation} 
\label{a-II-action}
\begin{aligned}
\fg \oplus \fg & \to \fg \oplus \fg \\
(R_1, R_2) & \mapsto 
((ag)R_2(ag)^{-1}, 
(g^{-1}a)R_1(g^{-1}a)^{-1}) . 
\end{aligned}
\end{equation}
That is, in terms of the presentation 
\begin{equation}
\label{a-II-presentation}
\begin{array}{ccccccccc}
0 & \to &  \fg & \to &  \fg/\fh \oplus \fg/\fh & \to & \rT_{[(H, gH)]}\cN & \to &  0 \\ 
&& S & \mapsto &   (S, g^{-1}Sg) &  & & &
\end{array}
\end{equation}
given by \eqref{Tg1Hg2H}, $(\sigma,a)_*$ is induced by \eqref{a-II-action}. 

We compute $\mult_1((\sigma,a)_*)$ by computing the multiplicity of 
the eigenvalue $1$ on the terms \mbox{$\fg/\fh \oplus \fg/\fh$} and  
$\fg$ in~\eqref{a-II-presentation} separately. 

\medskip \noindent
\textit{The $\fg/\fh \oplus \fg/\fh$ term}. 
Let $b = ag \in G$, which by assumption is in the image of the embedding $H \to G$. 
Note that $b^{-1} = g^{-1}a^{-1} = g^{-1}a$, hence~\eqref{a-II-action} 
can be written 
\begin{equation*}
(R_1, R_2) \mapsto (bR_2b^{-1}, b^{-1}R_1b)
\end{equation*}
The $+1$ eigenspace consists of $(R_1, R_2)$ such that $R_2 = b^{-1}R_1b$. 
Hence the induced map $\fg/\fh \oplus \fg/\fh \to \fg/\fh \oplus \fg/\fh$ 
has eigenvalue $1$ with multiplicity  $75 = \dim \fg/\fh$. 

\medskip \noindent
\textit{The $\fg$ term}. 
For $S \in \fg$ the map~\eqref{a-II-action} sends 
$(S, g^{-1}Sg) \mapsto (aSa^{-1}, g^{-1}(aSa^{-1})g)$. 
Hence the induced action on the term $\fg$ in~\eqref{a-II-presentation} 
is $S \mapsto aSa^{-1}$. 
By assumption $1 \neq a \in G$ is an involution, and so has 
type $\set{p,q}$ for some $p+q = 10$ and $p,q \geq 1$. 
By the observation from above, 
the map $\tfg \to \tfg$ given by $S \mapsto aSa^{-1}$ has eigenvalue $1$ 
with multiplicity $p^2 + q^2$; thus the corresponding map 
$\fg \to \fg$ has eigenvalue $1$ with multiplicity $p^2 + q^2 - 1$. 

\medskip
Combining the above, we conclude $\mult_1((\sigma,a)_*) = 76 - p^2 - q^2$ where 
the type $\set{p,q}$ of $a$ is one of the following:
$\set{9,1}, \set{8,2}, \set{7,3}, \set{6,4}, \set{5,5}$. 
Note that for $\set{p,q} = \set{9,1}$ this gives $\mult_1((\sigma,a)_*) = -6$, 
which is nonsense; so this case does not occur. 
Since $\dim \rT_{[(H, gH)]}\cN = 51$, we find  
$\tr((\sigma,a)_*) = 101 - 2(p^2+q^2)$. 
Plugging in $\set{p,q} = \set{8,2}, \set{7,3}, \set{6,4}, \set{5,5}$ 
gives the values in~\eqref{aut-involution-II} of the lemma. 
\end{proof} 

\subsection{Proof of Theorem~\ref{theorem-faithful-action}}
Suppose $\gamma \in \Aut_{\cN}(s)$ acts trivially on $\rT_{s}\cN$. 
If $\gamma$ is of type I, then $\gamma = 1$ by Lemma~\ref{lemma-aut-action-CY}\eqref{aut-action-I}. 
If $\gamma$ is of type II, then $\gamma^2$ is of type I, and hence $\gamma^2 = 1$ 
by the previous sentence. 
But then by Lemma~\ref{lemma-aut-involution}\eqref{aut-involution-II}, 
either $\gamma = 1$ or $\tr(\gamma_*) \in \set{1,-3,-15,-35}$. 
By assumption $\gamma_* = \id$ and hence $\tr(\gamma_*) = 51$, so 
we conclude $\gamma = 1$. \qed

\subsection{Proof of Proposition~\ref{proposition-tr-involution}}
This follows by combining Theorem~\ref{theorem-faithful-action} and 
Lemma~\ref{lemma-aut-involution}. 
(The case $\tr(\gamma_*) = 51$ corresponds to $\gamma = 1$.)  \qed

%%%%%%%%%%%%%%%%%%%%%%%%%%%%%%%%%%%%%%%%%%

\section{The double mirror involution} 
\label{section-tau}

We begin this section by showing that the operation of passing to the double mirror 
preserves smoothness of {\GPK} threefolds. 
Using this, in \S\ref{subsection-involution} we define the double mirror involution 
$\tau$ of the moduli stack $\cN$ of {\GPK} data. 
In \S\ref{subsection-derivative-tau} we compute the derivative of $\tau$. 

\subsection{Simultaneous smoothness} 
Let 
\begin{align*}
X & = \Gr_1 \cap \Gr_2 \subset \bP , \\
Y & = \Gr_1^\vee \cap \Gr_2^\vee \subset \bP^\vee , 
\end{align*}
be {\GPK} threefolds corresponding to isomorphisms 
$\phi_i \colon \wedge^2V \xrightarrow{\sim} W$, $i=1,2,$ as in \S\ref{section-intro}. 
We aim to show the following result, which is analogous to~\cite[Corollary~2.3]{pfaffian-grassmannian}. 

\begin{proposition}
\label{X-Y-smooth}
The variety $X$ is a smooth threefold if and only if the same is true of $Y$. 
\end{proposition}

\begin{remark}
If $X$ and $Y$ are of expected dimension, Proposition~\ref{X-Y-smooth} follows 
from Theorem~\ref{theorem-equivalence}, but we give a more direct proof below. 
\end{remark}

We start by recalling a basic fact about projective duality of $\Gr(2,V)$. 
\begin{lemma}
\label{lemma-projective-dual}
Let $x \in \Gr(2,V) \subset \bP(\wedge^2V)$ be a point corresponding to a 
$2$-plane $A \subset V$. 
Let $y \in \bP(\wedge^2V^\vee)$ be a point corresponding to a hyperplane $H \subset \bP(\wedge^2V)$. 
Then $H$ is tangent to $\Gr(2,V)$ at $x$ if and only if one of the following 
equivalent conditions hold:
\begin{enumerate}
\item \label{tangent-1} 
if $K \subset V$ denotes the kernel of the $2$-form on $V$ corresponding to $y$, 
then $A \subset K$. 
\item \label{tangent-2}
$y \in \Gr(2,V^\vee) \subset \bP(\wedge^2V^{\vee})$ and  
if $B \subset V^{\vee}$ is the corresponding $2$-plane with orthogonal 
$B^{\perp} = \ker(V \to B^{\vee})$, then $A \subset B^\perp$. 
\end{enumerate}
Moreover, in this case $K = B^{\perp}$. 
\end{lemma}

\begin{remark}
\label{remark-projective-dual}
The equivalence of $H$ being tangent to $\Gr(2,V)$ at $x$ with~\eqref{tangent-1} 
holds for $V$ of any dimension, while the equivalence with~\eqref{tangent-2} 
is special to the case $\dim V = 5$. 
Note that~\eqref{tangent-2} says in particular that 
the projective dual of $\Gr_i \subset \bP$ is $\Gr_i^\vee \subset \bP^{\vee}$, as the notation indicates. 
\end{remark}

We will deduce Proposition~\ref{X-Y-smooth} from an auxiliary result, 
which describes the loci in $X$ and $Y$ where the defining Grassmannians  
do not intersect transversally. 
Given $x \in X = \Gr_1 \cap \Gr_2$, we let $x_i \in \Gr_i$ be the 
two corresponding points, 
and we write $A_{x_i} \subset V$ for the corresponding $2$-planes. 
Similarly, for $y \in Y = \Gr_1^\vee \cap \Gr_2^\vee$ we let $y_i \in \Gr_i^\vee$ 
be the corresponding points, and write $B_{y_i} \subset V^\vee$ for the 
corresponding $2$-planes. 
Define $Z \subset X \times Y$ to be the locus of pairs $(x,y)$ such 
that $A_{x_i} \subset B_{y_i}^{\perp}$ for $i=1,2$, and let 
$\pr_X \colon Z \to X$ and $\pr_Y \colon Z \to Y$ be the two 
projections. 

\begin{lemma}
\label{lemma-non-transverse}
The following hold: 
\begin{enumerate}
\item \label{X-non-transverse}
A point $x \in X$ has $\dim \rT_{X,x} > 3$ if and only if 
it is in the image $\pr_X(Z)$. 
\item \label{Y-non-transverse}
A point $y \in Y$ has $\dim \rT_{Y,y} > 3$ if and only if 
it is in the image $\pr_Y(Z)$.  
\end{enumerate}
\end{lemma}

\begin{proof}
The condition $\dim \rT_{X,x} > 3$ is equivalent to $\rT_{\Gr_1, x}$ and 
$\rT_{\Gr_2,x}$ intersecting non-transversely in $\rT_{\bP,x}$, i.e. to the existence 
of a hyperplane in $\rT_{\bP,x}$ containing both $\rT_{\Gr_i, x}$, 
or equivalently to the existence of a projective hyperplane $H \subset \bP$ 
tangent to both $\Gr_i$ at $x$. But by Lemma~\ref{lemma-projective-dual}, 
the existence of such an $H$ is equivalent to the existence of a point $y \in Y$ 
such that $(x,y) \in Z$. This proves part~\eqref{X-non-transverse} of the lemma. 
Part~\eqref{Y-non-transverse} follows by symmetry (note that $Z$ can also be 
described as the locus of $(x,y)$ such that $B_{y_i} \subset A_{x_i}^\perp$, $i=1,2$). 
\end{proof}

\begin{proof}[Proof of Proposition~\ref{X-Y-smooth}]
If $X$ is a smooth threefold, then by Lemma~\ref{lemma-non-transverse}\eqref{X-non-transverse}
the correspondence $Z \subset X \times Y$ is empty. 
Hence $Y$, which a priori has dimension at least $3$, 
is in fact a smooth threefold by Lemma~\ref{lemma-non-transverse}\eqref{Y-non-transverse}. 
By symmetry, we conclude conversely that if $Y$ is a smooth threefold, then so is $X$. 
\end{proof}

\subsection{The double mirror involution}
\label{subsection-involution}
For any $(g_1, g_2) \in G \times G$, we have defined 
a Grassmannian intersection $X_{g_1,g_2}$ by~\eqref{Xg1g2}. 
Let 
\begin{equation}
\label{Yg1g2}
Y_{g_1,g_2} = (g_1 \Gr)^\vee \cap (g_2 \Gr)^\vee \subset \bP^\vee 
\end{equation}
be the corresponding double mirror. 

We can identify $Y_{g_1,g_2}$ with an explicit Grassmannian intersection in $\bP$, as follows. 
Fix from now on an isomorphism $V \cong V^\vee$ (or more explicitly, a basis for $V$). 
This induces an isomorphism $\wedge^2V \cong \wedge^2V^{\vee}$, 
and hence an isomorphism
\begin{equation*}
\theta \colon W \xrightarrow{\phi^{-1}} \wedge^2V \cong \wedge^2V^{\vee} 
\xrightarrow{(\phi^{-1})^*} W^\vee, 
\end{equation*}
which identifies $\Gr$ with $\Gr^\vee$. 
Given $g \in G$, its transpose is by definition the automorphism of $W$ given by 
$g^T = \theta^{-1} \circ g^* \circ \theta$, and its inverse transpose is 
$g^{-T} = (g^{-1})^T$. 
Here and below, we slightly abuse notation by not distinguishing between 
$g \in G = \PGL(W)$ and a lift of $g$ to $\GL(W)$. 

\begin{lemma}
\label{lemma-Yg1g2}
The isomorphism $\theta^{-1} \colon \bP^\vee \xrightarrow{\sim} \bP$ induces an isomorphism $Y_{g_1, g_2} \cong X_{g_1^{-T} \!, \, g_2^{-T}}$. 
\end{lemma}

\begin{proof}
For any $g \in G$, it follows from the definitions that 
\begin{equation*}
(g \Gr)^\vee = (g^{-1})^* \Gr^\vee \subset \bP^\vee. 
\end{equation*}
The result follows. 
\end{proof}

The involution 
\begin{align*}
G \times G & \to G \times G \\ 
(g_1,g_2) & \mapsto (g_1^{-T}, g_2^{-T})
\end{align*}
induces an involution $\ttau$ of $G/H \times G/H$. 
By Proposition~\ref{X-Y-smooth} combined with Lemma~\ref{lemma-Yg1g2},  
the involution $\ttau$ preserves the open subscheme $U \subset G \times G$ appearing 
in the definition~\eqref{equation-N} of the stack $\cN$, 
and corresponds to passing to the double mirror {\GPK} threefold on this locus. 
We denote by 
\begin{equation*}
\tau \colon \cN \to \cN 
\end{equation*} 
the induced involution, which we call the \emph{double mirror involution} of $\cN$.  

\subsection{Derivative of the double mirror involution} 
\label{subsection-derivative-tau}

In the following result, we use the terminology introduced directly after Lemma~\ref{lemma-presentation-TN}. 

\begin{lemma}
\label{lemma-derivative-tau}
Let $(g_1, g_2) \in G \times G$ be such that $[(g_1H, g_2H)] \in \cN$. 
Then the derivative 
\begin{equation*}
\rd_{[(g_1H, g_2H)]} \tau \colon \rT_{[(g_1H,g_2H)]}\cN \to \rT_{[(g_1^{-T}H,g_2^{-T}H)]}\cN 
\end{equation*}
is induced by the map 
\begin{align*}
\fg \oplus \fg & \to \fg \oplus \fg \\
(R_1, R_2) & \mapsto (-R_1^T, -R_2^T) . 
\end{align*}
\end{lemma}

\begin{proof}
This follows easily from the definitions and the observation 
\begin{equation*}
(1+\vepsilon R)^{-T} = 1 - \vepsilon R^T 
\end{equation*}
for $R \in \fg$. 
\end{proof}

%%%%%%%%%%%%%%%%%%%%%%%%%%%%%%%%%%%%%%%%%%

\section{Proof of Theorem~\ref{theorem-not-birational}}
\label{section-theorem-not-birational}

The moduli spaces and morphisms constructed in 
\S\ref{section-moduli} and \S\ref{section-tau}  
can be summarized by the diagram 
\begin{equation*}
\vcenter{
\xymatrix{
\cN \ar@(dl, ul)[]^{\tau} \ar[d]_{\pi_{\cN}} \ar[r]^{f} 
& \cM \ar[d]^{\pi_{\cM}} \\
N \ar[r] & M  
}
}
\end{equation*} 
where $f$ is the $\PGL$-parameterization of the moduli stack $\cM$ of 
{\GPK} threefolds by the moduli stack $\cN$ of {\GPK} data, 
$\tau$ is the double mirror involution, and 
$\pi_{\cN}$ and $\pi_{\cM}$ are coarse moduli spaces. 

Two {\GPK} threefolds are birational if and only if they are isomorphic, 
since they are Calabi--Yau of Picard number $1$. 
Recall that $f \colon \cN \to \cM$ induces an injection 
$|\cN(k)| \to |\cM(k)|$ by Lemma~\ref{lemma-f-points-auts}. 
(In fact $f$ is an open immersion by Theorem~\ref{theorem-moduli}, 
but we only need the weaker statement about points for the 
following argument.) 
Moreover, the locus $\cZ \subset \cN$ where the morphisms 
$\pi_{\cN} \circ \tau$ and $\pi_{\cN}$ agree is closed, 
because $N$ is separated by Lemma~\ref{lemma-N}.  
Hence Theorem~\ref{theorem-not-birational} is equivalent to the 
assertion that $\cZ$ does not coincide with $\cN$, i.e. 
that the morphisms $\pi_{\cN} \circ \tau$ and $\pi_{\cN}$ 
do not coincide. 
To prove Theorem~\ref{theorem-not-birational}, we will make a trace argument using 
Proposition~\ref{proposition-tr-involution} to show the following necessary condition for $\pi_{\cN} \circ \tau = \pi_{\cN}$ fails. 

\begin{lemma}
\label{lemma-fixed-point-tangent}
Let $s \in \cN$ be a point such that $\tau(s) = s$. 
If $\pi_{\cN} \circ \tau = \pi_{\cN}$, then the derivative 
\begin{equation*}
\rd_s \tau \colon \rT_{s}\cN \to \rT_{s} \cN 
\end{equation*} 
is contained in the image of the homomorphism 
$\Aut_{\cN}(s) \to \GL(\rT_{s}\cN)$. 
\end{lemma} 

\begin{proof}
Let $\Gamma = \Aut_{\cN}(s)$. 
By Luna's \'{e}tale slice theorem 
(see~\cite[Theorem 2.1]{alper-luna}), we may find an integral 
affine scheme $Y = \Spec(R)$ with a $\Gamma$-action, 
a point $y \in Y$, and an involution $\tau_Y \colon Y \to Y$, 
such that: 
\begin{enumerate}
\item \label{luna-1} 
$\tau_Y(y) = y$. 
\item \label{luna-2}
There is a $\Gamma$-equivariant isomorphism $\rT_y Y \cong \rT_s\cN$ 
under  which the maps $\rd_y \tau_Y$ and $\rd_s \tau$ are identified. 
\item \label{luna-3}
If $\pi_Y \colon Y \to Y \gitq \Gamma = \Spec(R^\Gamma)$ 
denotes the GIT quotient, then $\pi_Y \circ \tau_Y = \pi_Y$. 
\end{enumerate}
Consider the ring map $R^\Gamma \to R$ corresponding to $\pi_Y$. 
Passing the fraction fields, we obtain a Galois field extension 
$K(R)^{\Gamma} \to K(R)$. 
Since $\tau_Y$ restricts to an automorphism of $K(R)$ over 
$K(R)^{\Gamma}$ by~\eqref{luna-3}, we conclude that 
$\tau_Y$ coincides with the action of an element $\gamma \in \Gamma$ 
over the generic point of $Y$, and hence $\tau_Y$ coincides with 
the action of $\gamma$ on all of $Y$. 
Now the result follows from~\eqref{luna-2} by taking the derivative 
of $\tau_Y$ at $y$. 
\end{proof}

\begin{lemma}
\label{lemma-fixed-point}
There exists $g \in G$ with $g = g^{-T}$ such that $[(H, gH)] \in \cN$, i.e. 
such that $X_{1,g}$ is smooth. 
\end{lemma}

\begin{proof}
The orthogonal group $\rO(10)$ can be defined as a group scheme over
$\Spec(\bZ)$.  The construction of the scheme $X_{1,g}$ makes sense
for any $g \in \rO(10)$ (with arbitrary coefficients, not just over
$k$), and the locus
\begin{equation*}
U = \set{ g\in \rO(10) \st X_{1,g} \mbox{ is smooth} } 
\end{equation*}
is a Zariski open subset of $\rO(10)$.  The group scheme $\rO(10)$ is
smooth (and in particular flat) over $\Spec(\bZ)$, hence the
image of $U$ is an open subset of $\Spec(\bZ)$.  So if $U$ is
nonempty, then its image contains the generic point of
$\Spec(\bZ)$.  In other words, if for some prime
$p\in \bZ$ we find a matrix $g\in \rO(10)(\bF_p)$ such
that $X_{1,g}$ is smooth (over $\bF_p$), it follows that a
matrix with the same property exists over $\bQ \subset k$.  

We verified the existence of such a matrix in
$\rO(10)(\bF_{103})$ by an easy \texttt{Macaulay2} computation.
The code for this computation and the explicit matrix we found are
included in Appendix~\ref{appB}.
\end{proof}

Note that $[(H, gH)] \in \cN$ as in Lemma~\ref{lemma-fixed-point} satisfies 
$\tau([(H, gH)]) = [(H,g^{-T}H)] = [(H, gH)]$. 

\begin{lemma}
\label{lemma-trace-tau}
Let $g \in G$ be as in Lemma~\ref{lemma-fixed-point}. 
Then $\tr(\rd_{[(H, gH)]}\tau) = -1$. 
\end{lemma}

\begin{proof}
By Lemma~\ref{lemma-derivative-tau} the map 
$\rd_{[(H, gH)]}\tau \colon \rT_{[(H, gH)]}\cN \to \rT_{[(H, gH)]}\cN$ 
is induced by the map
\begin{equation} 
\label{tau-derivative} 
\begin{aligned}
\fg \oplus \fg & \to \fg \oplus \fg \\
(R_1, R_2) & \mapsto (-R_1^T, -R_2^T) . 
\end{aligned}
\end{equation} 
That is, in terms of the presentation 
\begin{equation}
\label{tau-presentation}
\begin{array}{ccccccccc}
0 & \to &  \fg & \to &  \fg/\fh \oplus \fg/\fh & \to & \rT_{[(H, gH)]}\cN & \to &  0 \\ 
&& S & \mapsto &   (S, g^{-1}Sg) &  & & &
\end{array}
\end{equation}
given by \eqref{Tg1Hg2H}, $\rd_{[(H, gH)]}\tau$ is induced by \eqref{tau-derivative}. 
On each copy of $\fg$ and $\fh$ appearing in~\eqref{tau-presentation}, the map  
induced by \eqref{tau-derivative} is $R \mapsto -R^T$. 
In general, given a vector space $L$, the trace of the map 
$\pgl(L) \to \pgl(L)$ given by $R \mapsto -R^{T}$ is $-\dim(L)+1$. 
Using this and additivity of traces, the result follows. 
\end{proof}

Now we can complete the proof of Theorem~\ref{theorem-not-birational}. 
Let $g \in G$ be as in Lemma~\ref{lemma-fixed-point} 
and let $s = [(H, gH)] \in \cN$. 
Then the involution $\rd_{s}\tau$ is not in the image of $\Aut_{\cN}(s) \to \GL(\rT_{s}\cN)$ 
by Proposition~\ref{proposition-tr-involution} and Lemma~\ref{lemma-trace-tau}. 
Hence   
the morphisms $\pi_{\cN} \circ \tau$ and $\pi_{\cN}$ do not coincide 
by Lemma~\ref{lemma-fixed-point-tangent}. \qed

%%%%%%%%%%%%%%%%%%%%%%%%%%%%%%%%%%%%%%%%%%

\section{Proof of Theorem~\ref{theorem-L-equivalence}}
\label{section-L-equivalence}

Let $X$ and $Y$ be {\GPK} double mirrors,
\begin{align*}
X & = \Gr_1 \cap \Gr_2 \subset \bP , \\
Y & = \Gr_1^\vee \cap \Gr_2^\vee \subset \bP^\vee, 
\end{align*}
corresponding to isomorphisms 
$\phi_i \colon \wedge^2V \xrightarrow{\sim} W$, $i=1,2$. 

As in the proof of \cite[Theorem 2.12]{borisov-zero-divisor}, we see that if $X$ 
is not birational to $Y$, then $[X] \neq [Y]$ in $\rK_0(\Var/k)$. 
Indeed, if $[X] = [Y]$ then $[X] = [Y] \bmod \bL$, and so $X$ is stably 
birational to $Y$ by \cite{larsen-lunts}. 
This means $X \times \bP^n$ is birational to $Y \times \bP^n$ for some $n$. 
But since $X$ and $Y$ are Calabi--Yau, 
they are the bases of the MRC fibrations \cite{KMM} of $X \times \bP^n$ and $Y \times \bP^n$, 
and hence birational. 

In particular, Theorem~\ref{theorem-not-birational} implies the second claim of Theorem~\ref{theorem-L-equivalence}. 

To prove the first claim of Theorem~\ref{theorem-L-equivalence}, we 
consider an incidence correspondence between $\Gr_1$ and $\Gr_2^{\vee}$.  
Namely, we consider the intersection 
\begin{equation*}
\rQ(\Gr_1, \Gr_2^{\vee}) = \rQ \times_{\bP \times \bPv} (\Gr_1 \times \Gr_2^{\vee}) 
\end{equation*}
of the canonical $(1,1)$ divisor $\rQ \subset \bP \times \bPv$ with the 
product $\Gr_1 \times \Gr_2^{\vee} \subset \bP \times \bPv$. 
We will calculate the class of $\rQ(\Gr_1, \Gr_2^{\vee})$ in 
$\rK_0(\Var/k)$ in two ways, using the two projections 
\begin{equation*}
\xymatrix{
& \rQ(\Gr_1, \Gr_2^{\vee}) \ar[dl]_{p_1} \ar[dr]^{p_2} & \\ 
\Gr_1 & & \Gr_2^{\vee}. 
}
\end{equation*}

Given $x \in \bP$, let $x_i = \phi_i^{-1}(x) \in \bP(\wedge^2V)$  
be the corresponding point for $i = 1,2$. 
Similarly, given $y \in \bPv$ let $y_i = \phi_i^{*}(y) \in \bP(\wedge^2V^{\vee})$ for $i=1,2$. 
Further, for a point $\omega$ in $\bP(\wedge^2V)$ or $\bP(\wedge^2V^{\vee})$, we write 
$\rk(\omega)$ for the rank of $\omega$ considered as a skew form (defined up to scalars); 
note that either $\rk(\omega) = 2$ or $\rk(\omega) = 4$. 
By definition we have 
\begin{alignat*}{2}
X & = \set{ x \in \bP \st \rk(x_1) = \rk(x_2) = 2 }  && \subset 
\Gr_1 = \set{ x \in \bP \st \rk(x_1) = 2 } ,  \\
Y  & = \set{ y \in \bPv \st \rk(y_1) = \rk(y_2) = 2 } && \subset \Gr_2^{\vee}  = \set{ y \in \bPv \st \rk(y_2) = 2 } , 
\end{alignat*} 
and hence also 
\begin{align*}
\Gr_1 \setminus X & = \set{ x \in \bP \st \rk(x_1) = 2, ~ \rk(x_2) = 4 } , \\
\Gr_2^{\vee} \setminus Y & = \set{ y \in \bPv \st \rk(y_1) = 4, ~ \rk(y_2) = 2} . 
\end{align*}
For $\omega$ in $\bP(\wedge^2V)$ or $\bP(\wedge^2V^{\vee})$, we let $H_{\omega}$ denote 
the corresponding hyperplane in the dual projective space. 
Then for $x \in \Gr_1$ and $y \in \Gr_2$ we have 
\begin{align*}
p_1^{-1}(x) & \cong H_{x_2} \cap \Gr(2,V^{\vee}) \subset \bP(\wedge^2V^{\vee}) , \\ 
p_2^{-1}(y) & \cong H_{y_1} \cap \Gr(2,V) \subset \bP(\wedge^2V). 
\end{align*}

Recall that a morphism of varieties $g \colon Z  \to S$ is called a \emph{piecewise trivial fibration  
with fiber $F$} if there is a finite partition $S = S_1 \sqcup S_2 \sqcup \cdots \sqcup S_n$, 
with each $S_i \subset Y$ a locally closed subset such that  
$g^{-1}(S_i) \cong S_i \times F$ as $S_i$-schemes. 

\begin{lemma} 
\label{lemma-fibrations} 
The following hold: 
\begin{enumerate}
\item The morphisms 
\begin{equation*}
p_1^{-1}(X) \to X  \quad \text{and} \quad p_2^{-1}(Y) \to Y
\end{equation*} 
are piecewise trivial fibrations, with fiber a hyperplane section of $\Gr(2,V)$ defined 
by a rank $2$ skew form. 
\item The morphisms 
\begin{equation*}
p_1^{-1}(\Gr_1 \setminus X) \to \Gr_1 \setminus X \quad \text{and} \quad 
p_2^{-1}(\Gr_2 \setminus Y) \to \Gr_2 \setminus Y 
\end{equation*}
are piecewise trivial fibrations, with fiber a hyperplane section of $\Gr(2,V)$ defined 
by a rank $4$ skew form. 
\end{enumerate}
\end{lemma}

\begin{proof}
By the above discussion, this follows as in~\cite[Lemma 3.3]{martin} from the fact that 
skew forms over $k$ can be put into one of the standard forms according to their rank. 
\end{proof}

Next we calculate the class of the fibers appearing in Lemma~\ref{lemma-fibrations}. 

\begin{lemma}
\label{lemma-lambda-section}
If $\omega \in \bP(\wedge^2V^{\vee})$, then 
\begin{equation*}
[H_{\omega} \cap \Gr(2,V)] = 
\begin{cases}
(\bL^2+\bL+1)(\bL^3+\bL^2+1) & \text{if } \rk(\omega) = 2 , \\ 
(\bL^2+1)(\bL^3+\bL^2+\bL+1) & \text{if } \rk(\omega) = 4 . 
\end{cases}
\end{equation*}
\end{lemma}

\begin{proof}
First assume $\rk(\omega) = 2$. Then the kernel $K \subset V$ of $\omega$ regarded as a skew 
form has $\dim K = 3$, and 
\begin{equation*}
H_{\omega} \cap \Gr(2,V) = \set{ A \in \Gr(2,V) \st A \cap K \neq 0 } . 
\end{equation*}
Consider the closed subset 
\begin{equation*}
Z = \set{ A \in \Gr(2,V) \st A \subset K } \subset H_{\omega} \cap \Gr(2,V), 
\end{equation*}
with open complement 
\begin{equation*}
U = \set{A \in \Gr(2,V) \st \dim(A \cap K) = 1 } \subset H_{\omega} \cap \Gr(2,V). 
\end{equation*}
Note that $Z \cong \bP^2$. 
Further, the natural morphism $U \to \bP(K) \cong \bP^2$ is a Zariski locally trivial fibration, whose fiber 
over $[ v ] \in \bP(K)$ is the complement in 
$\bP(V/\langle v \rangle) \cong \bP^3$ of $\bP(K/\langle v \rangle) \cong \bP^1$. 
Hence we have 
\begin{equation*}
[H_{\omega} \cap \Gr(2,V)] = [\bP^2] + [\bP^2]([\bP^3] - [\bP^1]) 
= (\bL^2+\bL+1)(\bL^3+\bL^2+1). 
\end{equation*}

Now assume $\rk(\omega) = 4$. 
Let $V_4 \subset V$ be a $4$-dimensional subspace such that the restriction of the form $\omega$ 
to $V_4$ has full rank. 
Consider the closed subset 
\begin{equation*}
Z = \set{ A \in H_{\omega} \cap \Gr(2,V) \st A \subset V_4 } \subset H_{\omega} \cap \Gr(2,V), 
\end{equation*}
with open complement 
\begin{equation*}
U = \set{ A \in H_{\omega} \cap \Gr(2,V) \st \dim(A \cap V_4) = 1 } \subset H_{\omega} \cap \Gr(2,V). 
\end{equation*}
Note that $Z$ is isomorphic to a smooth quadric hypersurface in $\bP^4$, whose class 
is well-known to be $[Z] = [\bP^3]$. 
Further, the natural morphism $U \to \bP(V_4) \cong \bP^3$ is a Zariski locally trivial fibration, 
whose fiber over $[ v ] \in \bP(V_4)$ consists of $[ v' ] \in \bP(V/\langle v \rangle )$ 
such that $\omega(v,v') = 0$ and $v' \notin V_4/\langle v \rangle$. That is, the fiber is isomorphic to 
the complement in 
$\bP(\langle v \rangle^{\perp_{\omega}}/\langle v \rangle) \cong \bP^2$ of 
$\bP( (\langle v \rangle^{\perp_{\omega}} \cap V_4)/\langle v \rangle) \cong \bP^1$, 
where $\langle v \rangle^{\perp_{\omega}} \subset V$ denotes the orthogonal of $\langle v \rangle \subset V$ 
with respect to the skew form $\omega$. 
Hence we have 
\begin{equation*}
[H_{\omega} \cap \Gr(2,V)] = [\bP^3] + [\bP^3]([\bP^2] - [\bP^1]) = (\bL^2 + 1)(\bL^3 + \bL^2 + \bL + 1), 
\end{equation*}
as claimed. 
\end{proof}

Now we can finish the proof of Theorem~\ref{theorem-L-equivalence}. 
Using the first projection $p_1$, we have 
\begin{equation*}
[\rQ(\Gr_1, \Gr_2^{\vee})] = [p_1^{-1}(X)] + [p_1^{-1}(\Gr_1 \setminus X)]. 
\end{equation*}
But if $g \colon Z \to S$ is a piecewise trivial fibration with fiber $F$, then 
$[Z] = [S][F]$. 
Hence using Lemmas~\ref{lemma-fibrations} and~\ref{lemma-lambda-section}, we find 
\begin{align*}
[\rQ(\Gr_1, \Gr_2^{\vee})] & = [X](\bL^2+\bL+1)(\bL^3+\bL^2+1) + 
([\Gr(2,V)] - [X]) (\bL^2+1)(\bL^3+\bL^2+\bL+1) , \\ 
& = [X] \bL^4 + [\Gr(2,V)] (\bL^2+1)(\bL^3+\bL^2+\bL+1) .
\end{align*}
The same argument applied to the second projection $p_2$ shows 
\begin{equation*}
[\rQ(\Gr_1, \Gr_2^{\vee})]  = [Y] \bL^4 + [\Gr(2,V)] (\bL^2+1)(\bL^3+\bL^2+\bL+1) . 
\end{equation*}
We conclude 
\begin{equation*}
([X]-[Y]) \bL^4 = 0 . \eqno\qed
\end{equation*}

%%%%%%%%%%%%%%%%%%%%%%%%%%%%%%%%%%%%%%%%%%

\appendix
\section{Borel--Weil--Bott computations}
\label{appendix}
The purpose of this appendix is to collect some coherent cohomology 
computations on Grassmannians and {\GPK} threefolds, which are invoked in the main text. 
The key tool is Borel--Weil--Bott, which we review in \S\ref{subsection-BWB}. 

\subsection{Borel--Weil--Bott}
\label{subsection-BWB}
For this subsection, we let $V$ denote an $n$-dimensional vector space over $k$ 
(in the rest of the paper $n=5$).  
The Borel--Weil--Bott Theorem for $\GL(V)$ allows us to compute the coherent cohomology of 
$\GL(V)$-equivariant bundles on a Grassmannian $\Gr(r,V)$ (in the rest of the paper we only 
need the case $r=2$). 
To state the result, we need some notation. 
Our exposition follows~\cite[\S{2.6}]{kuznetsov-HPD-lines}. 

The weight lattice of $\GL(V)$ is isomorphic to $\bZ^n$ via the map taking the $d$-th 
fundamental weight, i.e.\ the highest weight of $\wedge^dV$, to the sum of the first 
$d$ basis vectors of $\bZ^n$. Under this isomorphism, the dominant integral weights 
of $\GL(V)$ correspond to nonincreasing sequences of integers 
$\lambda = (\lambda_{1}, \dots, \lambda_{n})$. 
For such a $\lambda$, we denote by $\Sigma^{\lambda}V$ the corresponding irreducible 
representation of $\GL(V)$ of highest weight $\lambda$. 
The only facts we shall need about these representations are the following: 
\begin{enumerate}
\item If $\lambda = (1, \dots, 1, 0, \dots, 0)$ with the first $d$ entries equal to $1$, then
$\Sigma^{\lambda}V = \wedge^d V$. 

\item If $\lambda = (\lambda_1, \dots, \lambda_n)$ and $\mu =  (\lambda_1+m, \dots, \lambda_n+m)$ for some $m \in \bZ$, 
then there is an isomorphism of $\GL(V)$-representations
$\Sigma^{\mu}V \cong \Sigma^{\lambda}V \otimes \det(V)^{\otimes m}$. 

\item Given $\lambda = (\lambda_1, \dots, \lambda_n)$, set 
$\lambda^\vee = (-\lambda_{n}, -\lambda_{n-1}, \dots, -\lambda_{1})$. 
Then there is an isomorphism of $\GL(V)$-representations
$\Sigma^{\lambda^\vee}V \cong (\Sigma^{\lambda}V)^\vee$.
\end{enumerate}

The construction $V \mapsto \Sigma^{\lambda}V$ for a dominant integral weight $\lambda$ 
globalizes to vector bundles over a scheme, and the above identities continue to hold. 
We are interested in the case where the base scheme is the Grassmannian $\Gr(r,V)$. 
Denote by $\cU$ the tautological rank $r$ bundle on $\Gr(r,V)$, and by $\cQ$ the 
rank $n-r$ quotient of $V \otimes \cO_{\Gr(r,V)}$ by $\cU$, so that there is an exact sequence
\begin{equation*}
0 \rightarrow \cU \rightarrow V \otimes \cO \rightarrow \cQ \rightarrow 0.
\end{equation*}
Then every $\GL(V)$-equivariant bundle on $\Gr(r,V)$ is of the form 
$\Sigma^{\alpha}\cU^\vee \otimes \Sigma^{\beta} \cQ^\vee$ for some 
nonincreasing sequences of integers $\alpha \in \bZ^{r}$ and $\beta \in \bZ^{n-r}$.

The symmetric group $\rS_n$ %(the Weyl group of $\GL(V)$) 
acts on the weight lattice $\bZ^n$ by permuting 
the factors. Denote by $\ell: \rS_n \rightarrow \bZ$ the standard length function. 
We say $\lambda \in \bZ^n$ is \emph{regular} if all of its components are distinct; 
in this case, there is a unique $\sigma \in \rS_n$ such that $\sigma(\lambda)$ is a 
strictly decreasing sequence. Finally, let 
\begin{equation*}
\rho = (n,n-1, \dots, 2, 1) \in \bZ^n 
\end{equation*}
be the sum of the fundamental weights. 

The following result can be deduced from the usual statement of 
Borel--Weil--Bott by pushing forward equivariant line bundles on 
the flag variety to the Grassmannian. 
For a vector space $L$ and an integer $p$, we write $L[p]$ for 
the single-term complex of vector spaces with $L$ in degree $-p$. 

\begin{proposition}
\label{proposition-BWB}
Let the notation be as above. Let $\alpha \in \bZ^r$ and $\beta \in \bZ^{n-r}$ be 
nonincreasing sequences of integers, and let $\lambda = (\alpha, \beta) \in \bZ^n$ be 
their concatenation. If $\lambda + \rho$ is not regular, then 
\begin{equation*}
\RGamma(\Gr(r,V), \Sigma^{\alpha}\cU^\vee \otimes \Sigma^{\beta} \cQ^\vee) \cong 0
\end{equation*}
If $\lambda + \rho$ is regular and $\sigma \in \rS_n$ is the unique element such 
that $\sigma(\lambda + \rho)$ is a strictly decreasing sequence, then
\begin{equation*}
\RGamma(\Gr(r,V), \Sigma^{\alpha}\cU^\vee \otimes \Sigma^{\beta} \cQ^\vee) 
\cong \Sigma^{\sigma(\lambda + \rho) - \rho} V^\vee[-\ell(\sigma)].
\end{equation*}
\end{proposition}

For $r = 2$ we express the normal bundle of the Pl\"{u}cker embedding 
$\Gr(2,V) \subset \bP(\wedge^2V)$ in a form that is well-suited to applying Proposition~\ref{proposition-BWB}. 

\begin{lemma}
\label{lemma-normal-bundle}
The normal bundle of $\Gr(2,V) \subset \bP(\wedge^2V)$ satisfies 
\begin{equation*}
\rN_{\Gr(2,V)/\bP(\wedge^2V)} \cong \wedge^2 \cQ(1) \cong (\wedge^{n-4} \cQ^\vee)(2) 
\end{equation*}
where $n = \dim V$. 
\end{lemma} 

\begin{proof}
The normal bundle fits into a commutative diagram 
\begin{equation*}
\xymatrix{
& 0 \ar[d] & 0 \ar[d] & & \\
& \cU^\vee \otimes \cU \ar[d] \ar[r] & \cO \ar[d] & & \\
& \cU^{\vee} \otimes V \ar[d] \ar[r] & \wedge^2V \otimes \cO(1) \ar[r] \ar[d] & \cE \ar[d] \ar[r] & 0 \\ 
0 \ar[r] & \rT_{\Gr(2,V)} \ar[d] \ar[r] & \rT_{\bP(\wedge^2V)} \vert_{\Gr(2,V)} \ar[d] \ar[r] & \rN_{\Gr(2,V)/\bP(\wedge^2V)} \ar[r] & 0 \\
& 0 & 0 & & 
}
\end{equation*}
with exact rows and columns. 
Here, the map $\cU^\vee \otimes \cU \to \cO$ is given by evaluation. 
The map $\cU^\vee \otimes V \to \wedge^2V \otimes \cO(1)$ can be described as follows. 
Since $\det(\cU^\vee) \cong \cO(1)$ there is a natural isomorphism $\cU^\vee \cong \cU(1)$, 
and the map in question is the composition 
\begin{equation*}
\cU^\vee \otimes V \cong \cU \otimes V(1) \hookrightarrow V \otimes V \otimes \cO(1) 
\to \wedge^2V \otimes \cO(1). 
\end{equation*}
The sheaf $\cE$ is by definition the cokernel of this map. 
Due to the exact sequence 
\begin{equation*}
0 \to \cU \to V \otimes \cO \to \cQ \to 0
\end{equation*}
we therefore have an isomorphism
$\cE \cong (\wedge^2\cQ)(1)$.  
Hence also $\cE \cong (\wedge^{n-4} \cQ^\vee)(2)$ in view of the isomorphism $\det(\cQ) \cong \cO(1)$. 
It remains to note that $\cE \cong \rN_{\Gr(2,V)/\bP(\wedge^2V)}$ by the snake lemma. 
\end{proof}

\subsection{Computations on $\Gr$}
\label{subsection-BWB-Gr}
From now on, we assume $\dim V = 5$, 
fix an identification $\wedge^2V \cong W$, 
and let $\Gr \subset \bP$ denote the corresponding 
embedded Grassmannian $\Gr(2,V)$. 

\begin{lemma}
\label{lemma-IGr}
The ideal sheaf $\cI_{\Gr/\bP}$ of $\Gr \subset \bP$ admits a resolution of the form 
\begin{equation}
\label{resolution-IGr}
0 \to \cO(-5) \to V^\vee \otimes \cO(-3) \to V \otimes \cO(-2) \to \cI_{\Gr/\bP} \to 0 . 
\end{equation}
\end{lemma}

\begin{proof}
By regarding $\Gr \subset \bP$ as a Pfaffian variety, this follows from \cite{BE}  
(see \cite[Theorem~2.2]{kanazawa} for a statement of the result in the form that we apply it). 
\end{proof}

\begin{lemma}
\label{lemma-restriction-TPZ}
The restriction map $\rH^0(\bP, \rT_{\bP}) \to \rH^0(\Gr, \rT_{\bP}\vert_{\Gr})$ is an isomorphism. 
\end{lemma}

\begin{proof}
Taking cohomology of the exact sequence 
\begin{equation*}
0 \to \cI_{\Gr/\bP} \otimes \rT_\bP \to \rT_\bP \to \rT_\bP \vert_\Gr \to 0, 
\end{equation*}
we see it is enough to show $\rH^k(\bP, \cI_{\Gr/\bP} \otimes \rT_\bP) = 0$ for $k=0,1$. 
In fact, we claim the sheaf $\cI_{\Gr/\bP} \otimes \rT_\bP$ has no cohomology. 
Indeed, $\RGamma(\bP, \rT_\bP(-t)) \cong 0$ for $2 \leq t \leq 9$, as can be seen 
from the exact sequence 
\begin{equation*}
0 \to \cO \to W \otimes \cO(1) \to \rT_{\bP} \to 0,  
\end{equation*}
so the claim follows by tensoring the resolution~\eqref{resolution-IGr}
with $\rT_{\bP}$ and taking cohomology. 
\end{proof}

\begin{lemma}
\label{lemma-cohomology-Qid}
We have 
\begin{align*}
\RGamma(\Gr, \cQ(-t)) & \cong 
\begin{cases}
V[0] & t = 0, \\ 
0     & 1 \leq t \leq 5, \\
\Sigma^{(t-2,t-2,t-3,3,3)} V^\vee[-6]  & t \geq 6, 
\end{cases} \\
\RGamma(\Gr, \wedge^2\cQ(-t)) & \cong 
\begin{cases}
\wedge^2V[0] & t = 0, \\ 
0     & 1 \leq t \leq 5, \\
\Sigma^{(t-2,t-3,t-3,3,3)} V^\vee[-6]  & t \geq 6. 
\end{cases}
\end{align*}
\end{lemma}

\begin{proof}
Note that $\cQ(-t) \cong \Sigma^{(t,t,t-1)} \cQ^\vee$ 
and $\wedge^2\cQ(-t) \cong \Sigma^{(t,t-1,t-1)} \cQ^\vee$. 
Now the result follows from Proposition~\ref{proposition-BWB}.
\end{proof}

\begin{lemma}
\label{lemma-cohomology-N}
For $2 \leq t \leq 6$, we have 
$\RGamma(\Gr, \rN_{\Gr/\bP}(-t)) \cong 0$. 
\end{lemma}

\begin{proof}
Combine Lemmas~\ref{lemma-normal-bundle} and~\ref{lemma-cohomology-Qid}. 
\end{proof}

\subsection{Computations on a {\GPK} threefold} 
\label{subsection-BWB-KCY}
Let $X = \Gr_1 \cap \Gr_2$ be a {\GPK} threefold. 
We write $\cQ_i$ for the tautological rank $3$ quotient 
bundle on $\Gr_i$, and $\rN_i = \rN_{\Gr_i/\bP}$ for the normal bundle of $\Gr_i \subset \bP$. 

\begin{lemma}
For $i=1,2$, the ideal sheaf $\cI_{X/\Gr_i}$ of $X \subset \Gr_i$ admits a resolution 
of the form 
\begin{equation}
\label{resolution-IX}
0 \to \cO(-5) \to V^\vee \otimes \cO(-3) \to V \otimes \cO(-2) \to \cI_{X/\Gr_i} \to 0
\end{equation}
\end{lemma}

\begin{proof}
Analogously to Lemma~\ref{lemma-IGr}, 
this follows by regarding $X \subset \Gr_{i}$ as a Pfaffian 
variety. 
\end{proof}

\begin{lemma}
\label{lemma-class-X-Gr}
The class of $X$ in the Chow ring of $\Gr_i$ is $5H^3$, where 
$H$ denotes the Pl\"{u}cker hyperplane class. 
\end{lemma}

\begin{proof}
By~\eqref{resolution-IX} there is a resolution of $\cO_X$ on $\Gr_i$ of the form 
\begin{equation}
\label{resolution-OX}
0 \to \cO(-5) \to V^\vee \otimes \cO(-3) \to V \otimes \cO(-2) \to \cO \to \cO_{X} \to 0 . 
\end{equation}
The result follows by taking $\ch_3$. 
\end{proof}

\begin{lemma}
\label{lemma-Ni-vanishing}
For $t \geq 1$ we have 
\begin{equation*}
\rH^0(X, \cQ_i \vert_X(-tH)) = \rH^0(X, \wedge^2(\cQ_i \vert_X)(-tH)) = 0.
\end{equation*}
\end{lemma}

\begin{proof}
From~\eqref{resolution-OX} we get a resolution 
\begin{equation*}
0 \to \cQ_i(-(t+5)) \to V^\vee \otimes \cQ_i(-(t+3)) \to V \otimes \cQ_i(-(t+2)) \to \cQ_i(-tH) \to 
\cQ_i\vert_X(-tH) \to 0 . 
\end{equation*} 
Let $\cR_i^\bullet$ be the complex concentrated in degrees $[-3, 0]$ given by the 
first four terms, so that there is a quasi-isomorphism 
\begin{equation*}
\cR_i^\bullet \simeq \cQ_i\vert_X(-tH).
\end{equation*} 
Then the resulting spectral sequence
\begin{equation*}
E_1^{p,q} = \rH^q(X, \cR_i^p) \implies \rH^{p+q}(X, \cQ_i\vert_X(-tH)) 
\end{equation*}
combined with Lemma~\ref{lemma-cohomology-Qid} 
shows $\rH^0(X, \cQ_i\vert_X(-tH)) = 0$ for $t \geq 1$. 
The same argument also proves $\rH^0(X, \wedge^2(\cQ_i \vert_X)(-tH)) = 0$ for $t \geq 1$. 
\end{proof}

\begin{lemma}
\label{lemma-restriction-Qi}
The restriction maps
\begin{align*}
V \cong \rH^0(\Gr_i, \cQ_i) & \to \rH^0(X, \cQ_i \vert_X) , ~ i = 1,2, \\ 
W^\vee \cong \rH^0(\bP, \cO_{\bP}(1)) & \to \rH^0(X, \cO_X(1)) , 
\end{align*}
are isomorphisms. 
\end{lemma}

\begin{proof} 
Taking cohomology of the exact sequence
\begin{equation*}
0 \to \cI_{X/\Gr_i} \otimes \cQ_i \to \cQ_i \to \cQ_i \vert_X \to 0, 
\end{equation*}
the first claim follows from the vanishing $\RGamma(\Gr_i, \cI_{X/\Gr_i} \otimes \cQ_i) = 0$, 
which is a consequence of the resolution~\eqref{resolution-IX} combined with 
Lemma~\ref{lemma-cohomology-Qid}. The second claim is proved similarly. 
\end{proof} 

\begin{lemma}
\label{lemma-restriction-Ni}
The restriction maps $\rH^0(\Gr_i, \rN_i) \to \rH^0(X, \rN_i \vert_X)$, $i=1,2$,  
are isomorphisms.
\end{lemma}

\begin{proof}
Taking cohomology of the exact sequence 
\begin{equation*}
0 \to \cI_{X/\Gr_i} \otimes \rN_i \to \rN_i \to \rN_i \vert_X \to 0,  
\end{equation*}
we see it is enough to show $\rH^k(\Gr_i, \cI_{X/\Gr_i} \otimes \rN_i) = 0$ for $k=0,1$. 
In fact, we claim the sheaf $\cI_{X/\Gr_i} \otimes \rN_i$ has no cohomology. 
This follows by tensoring the resolution~\eqref{resolution-IX} 
with $\rN_i$ and applying Lemma~\ref{lemma-cohomology-N}.
\end{proof}

%%%%%%%%%%%%%%%%%%%%%%%%%%%%%%%%%%%%%%%%%%

\section{Macaulay2 computation}
\label{appB}

In this appendix we include the code used to find an orthogonal
$10\times 10$ matrix over the finite field $\bF_{103}$ such
that the corresponding {\GPK} threefold is smooth over $\bF_{103}$.

\bigskip

\begin{verbatim}
dotP = (v,u) -> (transpose(v)*u)_(0,0);
proj = (v,u) -> dotP(v,u)/dotP(u,u)*u;

p = 103; -- must equal 3 mod 4
q = 51; -- (p-1)/2
r = 26; -- (p+1)/4

isSquare = x->((x^q % 103) == 1);
sqRoot = x -> (x^r % 103);

kk = ZZ/p;

-- use Gram-Schmidt to find a random orthogonal matrix

v_1 = random(kk^10, kk^1);
while (not isSquare(dotP(v_1,v_1))) do v_1 = random(kk^10, kk^1);

for i from 2 to 10 do (
    test = true;
    while (test) do (
        v_i = random(kk^10, kk^1); 
        for j from 1 to i-1 do v_i = v_i-proj(v_i, v_j);
        test = not isSquare(dotP(v_i, v_i));
        )    
    )

for i from 1 to 10 do v_i = 1/sqRoot(dotP(v_i, v_i))*v_i;

-- T will be our candidate orthogonal matrix 

T = v_1 | v_2 | v_3 | v_4 | v_5 | v_6 | v_7 | v_8 | v_9 | v_10; 

R = kk[x01, x02, x03, x04, x12, x13, x14, x23, x24, x34];
M = matrix({{0, x01, x02, x03, x04}, 
        {-x01, 0, x12, x13, x14}, 
        {-x02, -x12, 0, x23, x24}, 
        {-x03, -x13, -x23, 0, x34}, 
        {-x04, -x14, -x24, -x34, 0}});
I = pfaffians(4,M);

V = genericMatrix(R, x01, 1, 10);
J = sub(I, V*T);

-- compute in the affine patch where x01 = 1

S = kk[x02, x03, x04, x12, x13, x14];
f = map(S, R, {1, x02, x03, x04, x12, x13, x14, 
        x02*x13-x03*x12, x02*x14-x04*x12, x03*x14-x04*x13});
CY = f J;
Jac = jacobian(CY);
Jac3 = minors(3, Jac);
Sing = Jac3+CY;
dim Sing -- answers -1, i.e., the empty set
\end{verbatim}
\bigskip

\noindent
The above code verifies, for a given choice of matrix $T$, that
the corresponding {\GPK} threefold is smooth in the affine patch where
$x_{01} = 1$.  However, one can directly verify that for the matrix
\[ \left (\begin{array}{cccccccccc}
-31 & 48 & -31 & 28 & -29 & 4 & 42 & -1 & -8 & 37  \\
2 & -1 & -36 & -26 & -12 & 9 & 6 & -7 & -14 & -14  \\
15 & 42 & 23 & 34 & 36 & -25 & -51 & 28 & 19 & -41  \\
-43 & -22 & -42 & -14 & -28 & 17 & 21 & 31 & -30 & 26  \\
-33 & -2 & 13 & -3 & -48 & 9 & 39 & 34 & -48 & -4  \\
34 & -26 & 33 & -25 & -22 & 45 & -33 & -26 & -23 & -43 \\
49 & -1 & 15 & -27 & -47 & -28 & -36 & 17 & -45 & -41 \\
-15 & 33 & -50 & -20 & -17 & 49 & 16 & 4 & -48 & 10 \\
-39 & -3 & -12 & 50 & -35 & 15 & -19 & -25 & 36 & 51 \\
-47 & -40 & -39 & 41 & 2 & -13 & -3 & 39 & 42 & 21 
\end{array}
\right )
\]
the same is true in all the patches $x_{ij} =1$, $0\leq i<j<5$.

%%%%%%%%%%%%%%%%%%%%%%%%%%%%%%%%%%%%%%%%%%

\providecommand{\bysame}{\leavevmode\hbox to3em{\hrulefill}\thinspace}
\providecommand{\MR}{\relax\ifhmode\unskip\space\fi MR }
\providecommand{\MRhref}[2]{%
  \href{http://www.ams.org/mathscinet-getitem?mr=#1}{#2}
}
\providecommand{\href}[2]{#2}

%%%%%%%%%%%%%%%%%%%%%%%%%%%%%%%%%%%%%%%%%%

\end{document}